\documentclass[11pt]{preprint}
\usepackage[full]{textcomp}
\usepackage[osf]{newtxtext} 
\usepackage{colortbl}
\usepackage{dsfont}
\usepackage{comment}
\usepackage[usenames,dvipsnames]{xcolor}

\usepackage{amssymb}
\usepackage{lmodern}
\usepackage{mathtools}

\usepackage{hyperref}
\usepackage{breakurl}
\usepackage{float}
\usepackage{aligned-overset}
\usepackage{mhenvs}
\usepackage{mhequ} 
\newcommand{\be}{\begin{equation*}}
\newcommand{\ee}{\end{equation*}}
\usepackage{mhsymb}
\usepackage{booktabs}
\usepackage{tikz,tkz-tab}
\tikzset{every picture/.style={line width=0.75pt}} 
\usetikzlibrary{matrix,decorations.pathreplacing, calc, positioning,fit,arrows.meta}

\usepackage{pgfplots}
\pgfplotsset{compat=1.18}
\usetikzlibrary{babel}

\usepackage{tcolorbox}
\usepackage{mathrsfs}
\usepackage[utf8]{inputenc}
\usepackage{longtable}
\usepackage{wrapfig}
\usepackage{subcaption}
\usepackage{mathrsfs}
\usepackage{epsfig}
\usepackage{microtype}
\usepackage{comment}
\usepackage{wasysym}
\usepackage{centernot}
\usepackage{enumitem}
\usepackage{bm}
\usepackage{stackrel}
\usepackage{graphicx}
\makeatletter
\newcommand{\globalcolor}[1]{%
  \color{#1}\global\let\default@color\current@color
}
\makeatother

\usetikzlibrary{calc}
\usetikzlibrary{decorations}
\usetikzlibrary{positioning}
\usetikzlibrary{shapes}
\usetikzlibrary{external}

\definecolor{blush}{rgb}{0.87, 0.36, 0.51}
	\definecolor{brightcerulean}{rgb}{0.11, 0.67, 0.84}
	\definecolor{greenryb}{rgb}{0.4, 0.69, 0.2}

\newif\ifdark
\darkfalse

\ifdark
\definecolor{darkred}{rgb}{0.9,0.2,0.2}
\definecolor{darkblue}{rgb}{0.7,0.3,1}
\definecolor{darkgreen}{rgb}{0.1,0.9,0.1}
\definecolor{franck}{rgb}{0,0.8,1}
\definecolor{pagebackground}{rgb}{.15,.21,.18}
\definecolor{pageforeground}{rgb}{.84,.84,.85}
\pagecolor{pagebackground}
\AtBeginDocument{\globalcolor{pageforeground}}
\definecolor{symbols}{rgb}{0,0.7,1}
\colorlet{connection}{red!80!black}
\colorlet{boxcolor}{blue!50}

\else

\definecolor{darkred}{rgb}{0.7,0.1,0.1}
\definecolor{darkblue}{rgb}{0.4,0.1,0.8}
\definecolor{darkgreen}{rgb}{0.1,0.7,0.1}
\definecolor{franck}{rgb}{0,0,1}
\definecolor{pagebackground}{rgb}{1,1,1}
\definecolor{pageforeground}{rgb}{0,0,0}
\colorlet{symbols}{blue!90!black}
\colorlet{connection}{red!30!black}
\colorlet{boxcolor}{blue!50!black}

\fi

\def\slash{\leavevmode\unskip\kern0.18em/\penalty\exhyphenpenalty\kern0.18em}
\def\dash{\leavevmode\unskip\kern0.18em--\penalty\exhyphenpenalty\kern0.18em}

\DeclareMathAlphabet{\mathbbm}{U}{bbm}{m}{n}

\DeclareFontFamily{U}{BOONDOX-calo}{\skewchar\font=45 }
\DeclareFontShape{U}{BOONDOX-calo}{m}{n}{
  <-> s*[1.05] BOONDOX-r-calo}{}
\DeclareFontShape{U}{BOONDOX-calo}{b}{n}{
  <-> s*[1.05] BOONDOX-b-calo}{}
\DeclareMathAlphabet{\mcb}{U}{BOONDOX-calo}{m}{n}
\SetMathAlphabet{\mcb}{bold}{U}{BOONDOX-calo}{b}{n}

\setlist{noitemsep,topsep=4pt,leftmargin=1.5em}

\DeclareMathAlphabet{\mathbbm}{U}{bbm}{m}{n}

\DeclareMathAlphabet{\mcb}{U}{BOONDOX-calo}{m}{n}
\SetMathAlphabet{\mcb}{bold}{U}{BOONDOX-calo}{b}{n}
\DeclareFontFamily{U}{mathx}{\hyphenchar\font45}
\DeclareFontShape{U}{mathx}{m}{n}{
      <5> <6> <7> <8> <9> <10>
      <10.95> <12> <14.4> <17.28> <20.74> <24.88>
      mathx10
      }{}
\DeclareSymbolFont{mathx}{U}{mathx}{m}{n}
\DeclareMathSymbol{\bigtimes}{1}{mathx}{"91}

\def\emptyset{{\centernot\ocircle}}

\providecommand{\figures}{false}
{ \ifthenelse{\equal{\figures}{false}} {#1}{$$ {\rm Figure \ missing !} $$} }{}

\usepackage{mathtools}

\tikzstyle{tinydots}=[dash pattern=on \pgflinewidth off \pgflinewidth]
\tikzstyle{superdense}=[dash pattern=on 4pt off 1pt]

\newcommand{\beq}{\begin{equation}}
\newcommand{\eeq}{\end{equation}}

\usepackage{empheq}

\def\${|\!|\!|}

\newcounter{theorem}

\newtheorem{ex}[theorem]{Example}

\newenvironment{DIFnomarkup}{}{} 

\theorembodyfont{\rmfamily}

\newfont{\indic}{bbmss12}

\def\Nabla_#1{\nabla_{\!#1}}

\makeatletter
\pgfdeclareshape{crosscircle}
{
  \inheritsavedanchors[from=circle] 
  \inheritanchorborder[from=circle]
  \inheritanchor[from=circle]{north}
  \inheritanchor[from=circle]{north west}
  \inheritanchor[from=circle]{north east}
  \inheritanchor[from=circle]{center}
  \inheritanchor[from=circle]{west}
  \inheritanchor[from=circle]{east}
  \inheritanchor[from=circle]{mid}
  \inheritanchor[from=circle]{mid west}
  \inheritanchor[from=circle]{mid east}
  \inheritanchor[from=circle]{base}
  \inheritanchor[from=circle]{base west}
  \inheritanchor[from=circle]{base east}
  \inheritanchor[from=circle]{south}
  \inheritanchor[from=circle]{south west}
  \inheritanchor[from=circle]{south east}
  \inheritbackgroundpath[from=circle]
  \foregroundpath{
    \centerpoint%
    \pgf@xc=\pgf@x%
    \pgf@yc=\pgf@y%
    \pgfutil@tempdima=\radius%
    \pgfmathsetlength{\pgf@xb}{\pgfkeysvalueof{/pgf/outer xsep}}%
    \pgfmathsetlength{\pgf@yb}{\pgfkeysvalueof{/pgf/outer ysep}}%
    \ifdim\pgf@xb<\pgf@yb%
      \advance\pgfutil@tempdima by-\pgf@yb%
    \else%
      \advance\pgfutil@tempdima by-\pgf@xb%
    \fi%
    \pgfpathmoveto{\pgfpointadd{\pgfqpoint{\pgf@xc}{\pgf@yc}}{\pgfqpoint{-0.707107\pgfutil@tempdima}{-0.707107\pgfutil@tempdima}}}
    \pgfpathlineto{\pgfpointadd{\pgfqpoint{\pgf@xc}{\pgf@yc}}{\pgfqpoint{0.707107\pgfutil@tempdima}{0.707107\pgfutil@tempdima}}}
    \pgfpathmoveto{\pgfpointadd{\pgfqpoint{\pgf@xc}{\pgf@yc}}{\pgfqpoint{-0.707107\pgfutil@tempdima}{0.707107\pgfutil@tempdima}}}
    \pgfpathlineto{\pgfpointadd{\pgfqpoint{\pgf@xc}{\pgf@yc}}{\pgfqpoint{0.707107\pgfutil@tempdima}{-0.707107\pgfutil@tempdima}}}
  }
}
\makeatother

\def\symbol#1{\textcolor{symbols}{#1}}

\def\decorate#1#2{
        \ifnum#2>0
    		\foreach \count in {1,...,#2}{
	       	let
				\p1 = (sourcenode.center),
                \p2 = (sourcenode.east),
				\n1 = {\x2-\x1},
				\n2 = {1mm},
				\n3 = {(1.3+0.6*(\count-1))*\n1},
				\n4 = {0.7*\n1}
			in 
        		node[rectangle,fill=symbols,rotate=30,inner sep=0pt,minimum width=0.2*\n2,minimum height=\n2] at ($(sourcenode.center) + (\n3,\n4)$) {}
				}
		\fi
        \ifnum#1>0
    		\foreach \count in {1,...,#1}{
	       	let
				\p1 = (sourcenode.center),
                \p2 = (sourcenode.east),
				\n1 = {\x2-\x1},
				\n2 = {1mm},
				\n3 = {(1.3+0.6*(\count-1))*\n1},
				\n4 = {0.7*\n1}
			in 
        		node[rectangle,fill=symbols,rotate=-30,inner sep=0pt,minimum width=0.2*\n2,minimum height=\n2] at ($(sourcenode.center) + (-\n3,\n4)$) {}
				}
		\fi
}

\tikzset{
    dectriangle/.style 2 args={
        triangle,
        alias=sourcenode,
        append after command={\decorate{#1}{#2}}
    },
    dectriangle/.default={0}{0},
}

\tikzset{
	cross/.style={path picture={ 
  		\draw[symbols]
			(path picture bounding box.south east) -- (path picture bounding box.north west) (path picture bounding box.south west) -- (path picture bounding box.north east);
		}},
root/.style={circle,fill=green!50!black,inner sep=0pt, minimum size=1.2mm},
        dot/.style={circle,fill=pageforeground,inner sep=0pt, minimum size=1mm},
        dotred/.style={circle,fill=pageforeground!50!pagebackground,inner sep=0pt, minimum size=2mm},
        var/.style={circle,fill=pageforeground!10!pagebackground,draw=pageforeground,inner sep=0pt, minimum size=3mm},
        kernel/.style={semithick,shorten >=2pt,shorten <=2pt},
        kernels/.style={snake=zigzag,shorten >=2pt,shorten <=2pt,segment amplitude=1pt,segment length=4pt,line before snake=2pt,line after snake=5pt,},
        rho/.style={densely dashed,semithick,shorten >=2pt,shorten <=2pt},
           testfcn/.style={dotted,semithick,shorten >=2pt,shorten <=2pt},
        renorm/.style={shape=circle,fill=pagebackground,inner sep=1pt},
        labl/.style={shape=rectangle,fill=pagebackground,inner sep=1pt},
        xic/.style={very thin,circle,draw=symbols,fill=symbols,inner sep=0pt,minimum size=1.2mm},
        g/.style={very thin,rectangle,draw=symbols,fill=symbols!10!pagebackground,inner sep=0pt,minimum width=2.5mm,minimum height=1.2mm},
        xi/.style={very thin,circle,draw=symbols,fill=symbols!10!pagebackground,inner sep=0pt,minimum size=1.2mm},
	xies/.style={very thin,rectangle,fill=green!50!black!25,draw=symbols,inner sep=0pt,minimum size=1.1mm},
	xiesf/.style={very thin,rectangle,fill=green!50!black,draw=symbols,inner sep=0pt,minimum size=1.1mm},
        xix/.style={very thin,crosscircle,fill=symbols!10!pagebackground,draw=symbols,inner sep=0pt,minimum size=1.2mm},
        X/.style={very thin,cross,rectangle,fill=pagebackground,draw=symbols,inner sep=0pt,minimum size=1.2mm},
	xib/.style={thin,circle,fill=symbols!10!pagebackground,draw=symbols,inner sep=0pt,minimum size=1.6mm},
	xie/.style={thin,circle,fill=green!50!black,draw=symbols,inner sep=0pt,minimum size=1.6mm},
	xid/.style={thin,circle,fill=symbols,draw=symbols,inner sep=0pt,minimum size=1.6mm},
	xibx/.style={thin,crosscircle,fill=symbols!10!pagebackground,draw=symbols,inner sep=0pt,minimum size=1.6mm},
	kernels2/.style={very thick,draw=connection,segment length=12pt},
	keps/.style={thin,draw=symbols,->},
	kepspr/.style={thick,draw=connection,->},
	krho/.style={thin,draw=symbols,superdense,->},
	krhopr/.style={thick,draw=connection,superdense},
	triangle/.style = { regular polygon, regular polygon sides=3},
	not/.style={thin,circle,draw=connection,fill=connection,inner sep=0pt,minimum size=0.5mm},
	diff/.style = {very thin,draw=symbols,triangle,fill=red!50!black,inner sep=0pt,minimum size=1.6mm},
	diff1/.style = {very thin,dectriangle={1}{0},fill=red!50!black,draw=symbols,inner sep=0pt,minimum size=1.6mm},
	diff2/.style = {very thin,dectriangle={1}{1},fill=red!50!black,draw=symbols,inner sep=0pt,minimum size=1.6mm},
		diffmini/.style = {very thin,rectangle,fill=black,draw=black,inner sep=0pt,minimum size=0.75mm},
	 kernelsmod/.style={very thick,draw=connection,segment length=12pt},
	 rec/.style = {very thin,rectangle,fill=black,draw=black,inner sep=0pt,minimum size=2mm},
	cerc/.style={very thin,circle,draw=black,fill=symbols,inner sep=0pt,minimum size=2mm},
	stars/.style={very thin,star,star points=6,star point ratio=0.5, draw=black,fill=red,inner sep=0pt,minimum size=0.7mm},
	>=stealth,
        }
        \tikzset{
root/.style={circle,fill=black!50,inner sep=0pt, minimum size=3mm},
        circ/.style={circle,fill=white,draw=black,very thin,inner sep=.5pt, minimum size=1.2mm},
        round1/.style={fill=white,outer sep = 0,inner sep=2pt,rounded corners=1mm,draw,text=black,thin,minimum size=1.2mm},
          circ1/.style={circle,fill=red!10,draw=red,very thin,inner sep=.5pt, minimum size=1.2mm},
        rect/.style={fill=white,outer sep = 0,inner sep=2pt,rectangle,draw,text=black,thin,minimum size=1.2mm},
        rect1/.style={fill=white,outer sep = 0,inner sep=2pt,rectangle,draw,text=black,thin,minimum size=1.2mm},
        round2/.style={fill=red!10,outer sep = 0,inner sep=2pt,rounded corners=1mm,draw,text=black,thin,minimum size=1.2mm},
       round3/.style={fill=blue!10,outer sep = 0,inner sep=2pt,rounded corners=1mm,draw,text=black,thin,minimum size=1.2mm}, 
        rect2/.style={fill=black!10,outer sep = 0,inner sep=2pt,rectangle,draw,text=black,thin,minimum size=1.2mm},
        dot/.style={circle,fill=black,inner sep=0pt, minimum size=1.2mm},
        dotred/.style={circle,fill=black!50,inner sep=0pt, minimum size=2mm},
        var/.style={circle,fill=black!10,draw=black,inner sep=0pt, minimum size=3mm},
        kernel/.style={semithick,shorten >=2pt,shorten <=2pt},
         diag/.style={thin,shorten >=4pt,shorten <=4pt},
        kernel1/.style={thick},
        kernels/.style={snake=zigzag,shorten >=2pt,shorten <=2pt,segment amplitude=1pt,segment length=4pt,line before snake=2pt,line after snake=5pt,},
		kernels1/.style={snake=zigzag,segment amplitude=0.5pt,segment length=2pt},
		rho1/.style={densely dotted,semithick},
        rho/.style={densely dashed,semithick,shorten >=2pt,shorten <=2pt},
           testfcn/.style={dotted,semithick,shorten >=2pt,shorten <=2pt},
           visible/.style={draw, circle, fill, inner sep=0.25ex},
        renorm/.style={shape=circle,fill=white,inner sep=1pt},
        labl/.style={shape=rectangle,fill=white,inner sep=1pt},
        xic/.style={very thin,circle,fill=symbols,draw=black,inner sep=0pt,minimum size=1.2mm},
        xi/.style={very thin,circle,fill=blue!10,draw=black,inner sep=0pt,minimum size=1.2mm},
	xib/.style={very thin,circle,fill=blue!10,draw=black,inner sep=0pt,minimum size=1.6mm},
	xie/.style={very thin,circle,fill=green!50!black,draw=black,inner sep=0pt,minimum size=1mm},
	xid/.style={very thin,circle,fill=symbols,draw=black,inner sep=0pt,minimum size=1.6mm},
	edgetype/.style={very thin,circle,draw=black,inner sep=0pt,minimum size=5mm},
	nodetype/.style={very thick,circle,draw=black,inner sep=0pt,minimum size=5mm},
	kernels2/.style={very thick,draw=connection,segment length=12pt},
clean/.style={thin,circle,fill=black,inner sep=0pt,minimum size=1mm},	not/.style={thin,circle,fill=symbols,draw=connection,fill=connection,inner sep=0pt,minimum size=0.8mm},
	>=stealth,
        }

\makeatletter
\def\DeclareSymbol#1#2#3{%
	\expandafter\gdef\csname MH@symb@#1\endcsname{\tikzsetnextfilename{symbol#1}%
	\tikz[baseline=#2,scale=0.15,draw=symbols,line join=round]{#3}}%
	\expandafter\gdef\csname MH@symb@#1s\endcsname{\scalebox{0.75}{\tikzsetnextfilename{symbol#1}%
	\tikz[baseline=#2,scale=0.15,draw=symbols,line join=round]{#3}}}%
	\expandafter\gdef\csname MH@symb@#1ss\endcsname{\scalebox{0.65}{\tikzsetnextfilename{symbol#1}%
	\tikz[baseline=#2,scale=0.15,draw=symbols,line join=round]{#3}}}%
	}
\def\<#1>{\ifthenelse{\boolean{mmode}}{\mathchoice{\csname MH@symb@#1\endcsname}{\csname MH@symb@#1\endcsname}{\csname MH@symb@#1s\endcsname}{\csname MH@symb@#1ss\endcsname}}{\csname MH@symb@#1\endcsname}}
\makeatother

 \def\1{\mathbf{\symbol{1}}}

\DeclareMathAlphabet{\mathpzc}{OT1}{pzc}{m}{it}

\def\eqref#1{(\ref{#1})}

\makeatletter 
\newcommand*{\bigcdot}{}
\DeclareRobustCommand*{\bigcdot}{%
  \mathbin{\mathpalette\bigcdot@{}}%
}
\newcommand*{\bigcdot@scalefactor}{.5}
\newcommand*{\bigcdot@widthfactor}{1.15}
\newcommand*{\bigcdot@}[2]{%
  \sbox0{$#1\vcenter{}$}
  \sbox2{$#1\cdot\m@th$}%
  \hbox to \bigcdot@widthfactor\wd2{%
    \hfil
    \raise\ht0\hbox{%
      \scalebox{\bigcdot@scalefactor}{%
        \lower\ht0\hbox{$#1\bullet\m@th$}%
      }%
    }%
    \hfil
  }%
}
\makeatother

\def\act{\bigcdot}

\tcbset
{colframe=boxcolor,colback=symbols!7!pagebackground,coltext=pageforeground,
fonttitle=\bfseries,nobeforeafter,center title,size=fbox,boxsep=1.5pt,
top=0mm,bottom=0mm,boxsep=0mm,tcbox raise base}

\def\two{{\<generic>\kern0.05em\<genericb>}}
\def\twoI{{\<Ito>\kern0.05em\<Itob>}}

\def\mail#1{\burlalt{#1}{mailto:#1}}

\usepackage{thmtools}

\definecolor{greenline}{RGB}{60,130,80}
\definecolor{blueconn}{RGB}{0,90,255}
\definecolor{rednote}{RGB}{220,50,47}

\tikzset{
cutatom/.style={
circle,
draw=rednote,
fill=white,
very thick,
minimum size=7pt,
inner sep=0pt
},
atom/.style={
circle,
draw=black,
fill=white,
thick,
minimum size=7pt,
inner sep=0pt
}
}
\begin{document}

\title{Kruskal-style algorithm for Boltzmann equation molecule reduction}

\author{Yvain Bruned, Valentin Clarisse}
\institute{ 
	Universite de Lorraine, CNRS, IECL, F-54000 Nancy, France
	\\
	Email:\ \begin{minipage}[t]{\linewidth}
		\mail{yvain.bruned@univ-lorraine.fr}
		\\
		\mail{valentin.clarisse@univ-lorraine.fr}.
\end{minipage}}

\maketitle 

\begin{abstract}
We are interested in the molecule reduction algorithm used by Deng, Hani and Ma. This algorithm allows them to establish a rigidity theorem, which plays a central role in the long-time derivation of the Boltzmann equation using the hard-sphere dynamics. In the present article, we show that this algorithm is a graph traversal algorithm of Kruskal type, and we prove that it constructs a Kruskal spanning tree of the input molecule.
\end{abstract}

\setcounter{tocdepth}{2}
\tableofcontents
\section{Introduction}

Let $d\geqslant 2$, $\varepsilon>0$, we define the hard-sphere dynamics for $N$ spheres of diameter $\varepsilon$ in dimension $d$. A configuration of the system is given by:
$$
\mathbf{z}_N=(z_1,\ldots,z_N)
=
((x_i,v_i))_{1\leqslant i\leqslant N}
\in\mathcal{D}_N,
$$
where $x_i,v_i\in\mathbb{R}^d$ denote respectively the position and the
velocity of the $i$-th particle, and:
$$
\mathcal{D}_N
=
\left\{
\mathbf{z}_N\in\mathbb{R}^{2dN},
\lvert x_i-x_j\rvert\geqslant\varepsilon
\text{ for every }i\neq j
\right\}.
$$
The hard-sphere dynamics is defined by free transport between collision
times:
$$
\dfrac{\mathrm{d}}{\mathrm{d}t}(x_i,v_i)=(v_i,0).
$$
If two particles $i$ and $j$ collide at time $t$, ie $\lvert x_i(t)-x_j(t)\rvert=\varepsilon$, the velocities after the collision are given by
$$
\begin{cases}
v_i(t^+)
&=
v_i(t)
-
\left((v_i(t)-v_j(t))\cdot\omega\right)\omega,\\
v_j(t^+)
&=
v_j(t)
+
\left((v_i(t)-v_j(t))\cdot\omega\right)\omega,
\end{cases}
$$
where
$$
\omega=\dfrac{x_i(t)-x_j(t)}{\varepsilon}
\in\mathbb{S}^{d-1}.
$$

We denote by $\mathcal{H}_N(t)$ the associated flow. It is
well known that, outside a Lebesgue-negligible subset of $\mathcal{D}_N$,
this flow is well defined for all times and preserves the Lebesgue measure;
see~\cite{A75}. We will ignore this negligible set in the following.
As in~\cite{DHM25a}, we consider random initial configurations in the grand
canonical ensemble. We set
$$
\mathcal{D}
=
\bigcup_{N\in\mathbb{N}}\mathcal{D}_N.
$$
Let $f_0:\mathbb{R}^{2d}\to\mathbb{R}_+$ be a probability density. The
initial density on $\mathcal{D}_N$ is
$$
W_{0,N}(\mathbf{z}_N)
=
\frac{1}{Z}
\varepsilon^{-(d-1)N}
\prod_{i=1}^N f_0(z_i)
\mathds{1}_{\mathcal{D}_N}(\mathbf{z}_N),
$$
where $Z$ is a normalisation constant. Hence, for every measurable set
$A\subset\mathcal{D}_N$,
$$
\mathbf{P}\left(\mathbf{z}(0)\in A\right)
=
\frac{1}{N!}
\int_A
W_{0,N}(\mathbf{z}_N)
\mathrm{d}\mathbf{z}_N
$$
where $\mathbf{z}(0)$ is the random initial data. The factor $\varepsilon^{-(d-1)N}$ corresponds to the Boltzmann-Grad
scaling
$$
\mathbf{E}(N)\varepsilon^{d-1}\simeq 1,
$$
so that the expected number of particles is of order
$\varepsilon^{-(d-1)}$.
For $t\geqslant0$, we write
$$
W_N(t)=\mathcal{S}_N(t)W_{0,N}
$$
for the density transported by the hard-sphere flow, where $\mathcal{S}_N(t)$ is the pullback of $W_{0,N}$ using the flow of $\mathcal{H}_N(t)$. The rescaled
$s$-particle correlation function is defined by
$$
f_s(t,\mathbf{z}_s)
=
\varepsilon^{(d-1)s}
\sum_{n=0}^{+\infty}
\frac{1}{n!}
\int_{\mathbb{R}^{2dn}}
W_{s+n}(t,\mathbf{z}_{s+n})
\mathrm{d}z_{s+1}\cdots\mathrm{d}z_{s+n}.
$$

The effective equation expected in the limit $\varepsilon\to0$ is the
Boltzmann equation. We say that
$f:\mathbb{R}_+\times\mathbb{R}^{2d}\to\mathbb{R}_+$ is a solution to the
hard-sphere Boltzmann equation with initial datum $f_0$ if
$$
\begin{cases}
\displaystyle
(\partial_t+v\cdot\nabla_x)f
=
\int_{\mathbb{R}^d}
\int_{\mathbb{S}^{d-1}}
\left((v-v_1)\cdot\omega\right)_+
\left(
f(x,v')f(x,v_1')
-
f(x,v)f(x,v_1)
\right)
\mathrm{d}\omega\mathrm{d}v_1,
\\
f(0,\cdot)=f_0,
\end{cases}
$$
where
$$
v'
=
v-\left((v-v_1)\cdot\omega\right)\omega
\quad\text{and}\quad
v_1'
=
v_1+\left((v-v_1)\cdot\omega\right)\omega.
$$

The derivation of this equation from the hard-sphere dynamics goes back
to the work of Grad~\cite{G58} and to the Boltzmann-Grad scaling.
The first rigorous results were obtained by Cercignani~\cite{C72}
and, shortly after, by Lanford~\cite{L75}. Lanford proved the
convergence of the hard-sphere dynamics towards the Boltzmann equation,
together with propagation of chaos, on a sufficiently short time
interval. Since then, this result has been refined in several directions;
see for instance~\cite{S91,CIP94,GST14,PSS14}. More recently, fluctuations
around the Boltzmann dynamics and the corresponding large deviations have
also been studied in~\cite{BGSS23}.

The restriction to short times remained the main difficulty. Long-time
results were known in some particular situations, for instance for
perturbations of vacuum~\cite{IP89,D18} or tagged particles close to
equilibrium~\cite{BGSR16}, but the general off-equilibrium problem
remained open.

In~\cite{DHM25a}, Deng, Hani and Ma obtained the long-time derivation of the
Boltzmann equation from the hard-sphere dynamics. More precisely, they
prove convergence on any interval on which the Boltzmann solution stays
sufficiently regular. We recall their main theorem below.
\begin{theorem}[Deng-Hani-Ma]
\label{main_theorem_DHM25}
Let $d\geqslant2$, $\beta>0$, and let
$f_0:\mathbb{R}^{2d}\to\mathbb{R}_+$ be a probability density. Suppose
that the solution $f$ of the Boltzmann equation exists on
$[0,t_{\mathrm{fin}}]$ and satisfies
$$
\sup_{t\in[0,t_{\mathrm{fin}}]}
\left\lVert
e^{2\beta\lvert v\rvert^2}f(t)
\right\rVert_{L^\infty_{x,v}}
\leqslant A<+\infty,
$$
and
$$
\lVert f_0\rVert_{\mathrm{Bol}^{2\beta}}
+
\lVert\nabla_xf_0\rVert_{\mathrm{Bol}^{2\beta}}
\leqslant B_0<+\infty,
$$
where
$$
\lVert f\rVert_{\mathrm{Bol}^{\beta}}
=
\sum_{k\in\mathbf{Z}^d}
\sup_{\substack{\lvert x-k\rvert\leqslant1\\v\in\mathbb{R}^d}}
e^{\beta\lvert v\rvert^2}\lvert f(x,v)\rvert.
$$

Consider the hard-sphere system with diameter $\varepsilon$, random initial
data given by the grand canonical ensemble above and under the
Boltzmann-Grad scaling. For $\varepsilon$ small enough depending on
$d,t_{\mathrm{fin}},\beta,A$ and $B_0$, there exists
$\theta=\theta(d)>0$ such that, uniformly for
$t\in[0,t_{\mathrm{fin}}]$ and $s\leqslant\lvert\log\varepsilon\rvert$,
$$
\left\lVert
f_s(t,\mathbf{z}_s)
-
\prod_{j=1}^s
f(t,z_j)
\mathds{1}_{\mathcal{D}_s}(\mathbf{z}_s)
\right\rVert_{L^1(\mathbb{R}^{2ds})}
\leqslant
\varepsilon^\theta.
$$
\end{theorem}

This result extends the derivation beyond the short time appearing in
Lanford's theorem. An account of the result and of the main ideas of its
proof can be found in the Bourbaki seminar~\cite{BGSS26}; see also the
expository notes~\cite{DEx1,DEx2}. The method was subsequently extended
in~\cite{DHM25b}, where Deng, Hani and Ma derive fundamental equations of
fluid mechanics from hard-sphere dynamics through Boltzmann's kinetic
theory.

Several ideas used in~\cite{DHM25a} come from the works of Deng and Hani
on wave kinetic equations~\cite{DH23a,DH23b,DH26}, where similar
combinatorial difficulties appear when one tries to reach the kinetic
timescale. Related recent developments include the one-dimensional wave
kinetic theory studied in~\cite{V25} and the derivation of the wave kinetic
equation for the $\beta$-FPUT system in~\cite{VW26}.

Let us summarise the main ideas of the proof of
Theorem~\ref{main_theorem_DHM25}. The first step is to divide the time
interval into layers
$$
[0,t_{\mathrm{fin}}]
=
\bigcup_{\ell=1}^{L}
[(\ell-1)\tau,\ell\tau],
\qquad
\tau=\frac{t_{\mathrm{fin}}}{L},
$$
with $\tau$ sufficiently small. On each time layer, one can use a local
expansion. The problem is then to propagate the estimates from one layer
to the next. At time $\ell\tau$, the particles are no longer independent,
as correlations have been generated by the previous collisions.

These correlations are described using cumulants. Following the cluster
expansion approach for hard-sphere systems~\cite{PS17,BGSS22}, one writes
$$
f_s(t,\mathbf{z}_s)
=
\prod_{j=1}^s f^A(t,z_j)
+
\sum_{\emptyset\neq H\subset\{1,\ldots,s\}}
\left(
\prod_{j\in\{1,\ldots,s\}\setminus H}f^A(t,z_j)
\right)
E_H(t,\mathbf{z}_H),
$$
where $f^A$ is the approximately factorized contribution and $E_H$ are the
cumulants. The aim is then to prove that $f^A$ converges towards the
solution of the Boltzmann equation and that the cumulants remain small,
with estimates of the form
$$
\lVert E_H(\ell\tau)\rVert_{L^1}
\leqslant
\varepsilon^{\alpha\lvert H\rvert}
$$
for some $\alpha>0$.

A direct iteration of these estimates from one layer to the next is not
sufficient. The authors introduce instead a partial time expansion: the
approximately factorized terms are not expanded, whereas the cumulant
terms are recursively expanded. The resulting terms encode the collision
histories responsible for the correlations.

These collision histories can be represented by decorated graphs called
\emph{molecules}. The edges describe pieces of particle trajectories and
the vertices, called \emph{atoms}, represent collisions. The layer of an
atom records the time interval in which the corresponding collision takes
place. We recall these objects and their analytical interpretation in
Section~\ref{Sec::2}. The cumulants are reduced to estimates of the form
$$
\lvert E_H(\ell\tau)\rvert
\lesssim
\sum_M
\lvert\mathcal{IN}_M\rvert,
$$
where $\lvert\mathcal{IN}_M\rvert$ is the normalised contribution
associated with a molecule $M$. The main problem is then to estimate
$\lvert\mathcal{IN}_M\rvert$.
In~\cite{DHM25a}, this quantity is rewritten in terms of an integral
operator $I_M$. The choice of the order of integration is encoded through
a cutting procedure on the molecule. Each cutting step removes some atoms
and bonds, and the procedure is iterated until the molecule is decomposed
into elementary submolecules.

The elementary submolecules are classified according to the analytical
estimate they provide. In particular, $\{3\}$ molecules are normal,
$\{33\}$ molecules are good and $\{4\}$ molecules are bad. Good molecules
give an additional gain in $\varepsilon$, whereas bad molecules produce
a loss. The goal of the cutting algorithm is therefore to produce enough
good molecules while keeping the number of bad molecules small. See Table \ref{tab:submolecule-classification} for the classification of elementary submolecules.

The algorithm introduced in~\cite{DHM25a} is rather sophisticated. A
multi-layer molecule is first reduced to a two-layer molecule. The
two-layer case is then treated by several algorithms, among which
UP and the toy models I, I+, II and III.
The succession of these procedures is determined by the degrees of the
atoms and by the geometry of the molecule. Here is an example of molecule reduction using part of the UP algorithm:
\begin{figure}[H]
\centering
\begin{tikzpicture}[
    line cap=round,
    line join=round,
    edge/.style={black,thick},
    rededge/.style={rednote,thick},
    vertex/.style={
        circle,
        draw=black,
        fill=white,
        inner sep=1.2pt,
        line width=0.8pt
    },
    redvertex/.style={
        circle,
        draw=rednote,
        fill=white,
        inner sep=1.2pt,
        line width=0.8pt
    }
]

\begin{scope}[xshift=0cm]
\coordinate (BL) at (-1,-1);
\coordinate (L)  at (0,0);
\coordinate (T)  at (1,1);
\coordinate (R)  at (2,0);
\coordinate (B)  at (1,-1);

\draw[edge] (L)--(T)--(R)--(B)--(L);
\draw[edge] (L)--(BL);

\draw[edge] (BL) -- ++(-0.18, 0.18);
\draw[edge] (BL) -- ++(-0.18,-0.18);
\draw[edge] (BL) -- ++( 0.18,-0.18);

\draw[edge] (L) -- ++(-0.18, 0.18);

\draw[edge] (T) -- ++(-0.18, 0.18);
\draw[edge] (T) -- ++( 0.18, 0.18);

\draw[edge] (R) -- ++( 0.18,-0.18);
\draw[edge] (R) -- ++( 0.18, 0.18);

\draw[edge] (B) -- ++(-0.18,-0.18);
\draw[edge] (B) -- ++( 0.18,-0.18);

\foreach \P in {BL,L,T,R,B}{
    \node[vertex] at (\P) {};
}

\node[below left=4pt] at (BL) {$\mathfrak{n}_1$};
\node[left=4pt]       at (L)  {$\mathfrak{n}_2$};
\node[above left=4pt] at (T)  {$\mathfrak{n}_3$};
\node[right=4pt]      at (R)  {$\mathfrak{n}_5$};
\node[below=4pt]      at (B)  {$\mathfrak{n}_6$};
\end{scope}

\begin{scope}[xshift=6.8cm]
\coordinate (BL) at (-1,-1);
\coordinate (L)  at (0,0);
\coordinate (T)  at (1,1);
\coordinate (TR) at (2,2);
\coordinate (R)  at (2,0);
\coordinate (B)  at (1,-1);

\draw[edge]    (L)--(T)--(R)--(B)--(L);
\draw[edge]    (L)--(BL);
\draw[rededge] (T)--(TR);

\draw[edge] (BL) -- ++(-0.18, 0.18);
\draw[edge] (BL) -- ++(-0.18,-0.18);
\draw[edge] (BL) -- ++( 0.18,-0.18);

\draw[edge] (L) -- ++(-0.18, 0.18);

\draw[edge] (T) -- ++(-0.18, 0.18);

\draw[rededge] (TR) -- ++(-0.18, 0.18);
\draw[rededge] (TR) -- ++( 0.18,-0.18);
\draw[rededge] (TR) -- ++( 0.18, 0.18);
\node[font=\scriptsize,text=rednote] at (2.24,1.82) {$\times$};

\draw[edge] (R) -- ++( 0.18,-0.18);
\draw[edge] (R) -- ++( 0.18, 0.18);

\draw[edge] (B) -- ++(-0.18,-0.18);
\draw[edge] (B) -- ++( 0.18,-0.18);

\node[vertex]    at (BL) {};
\node[vertex]    at (L)  {};
\node[vertex]    at (T)  {};
\node[redvertex] at (TR) {};
\node[vertex]    at (R)  {};
\node[vertex]    at (B)  {};

\node[below left=4pt]  at (BL) {$\mathfrak{n}_1$};
\node[left=4pt]        at (L)  {$\mathfrak{n}_2$};
\node[above left=4pt]  at (T)  {$\mathfrak{n}_3$};
\node[above right=4pt] at (TR) {$\mathfrak{n}_4$};
\node[right=4pt]       at (R)  {$\mathfrak{n}_5$};
\node[below=4pt]       at (B)  {$\mathfrak{n}_6$};
\end{scope}

\end{tikzpicture}
\caption{The atom $\mathfrak{n}_4$ has one bottom fixed end denoted by the symbol $\times$. In the UP algorithm, we choose  $\mathfrak{n}=\mathfrak{n}_4$ and $S_{\mathfrak{n}}=\{\mathfrak{n}_4,\mathfrak{n}_3,\mathfrak{n}_2,\mathfrak{n}_5,\mathfrak{n}_1,\mathfrak{n}_6\}$, cutting $\mathfrak{m}=\mathfrak{n}_4$.}
\end{figure}
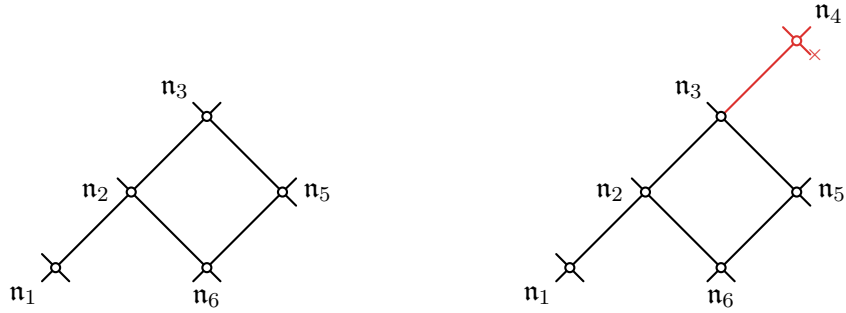
The point is that the UP algorithm has two orderings:
\begin{itemize}
\item globally, choose the lowest degree-$3$ atom $\mathfrak{n}$;
\item once $\mathfrak{n}$ is chosen, work inside its descendant set $S_\mathfrak{n}$, but cut that set from top to bottom.
\end{itemize}

Hence, in the particular case of the UP algorithm, we build a spanning tree of the molecule using a local Prim-like algorithm.

A similar situation appears in the derivation of the wave kinetic
equation. In~\cite{BC26}, the molecule reduction algorithm of Deng and
Hani (see~\cite{DH23a}) was interpreted as a graph traversal algorithm of
Kruskal type. More precisely, the reduction algorithm was shown to
construct a spanning tree of the input molecule by keeping, at every step,
a maximal number of edges without creating cycles.

The use of trees and forests to organise perturbative expansions also has
a long history in renormalisation and constructive quantum field theory;
see for instance~\cite{Z69,R91,RW14}. In particular, spanning trees can be
used to reorganise families of Feynman graphs, a point of view which is
especially explicit in~\cite{RW14}.
We consider a two-layer molecule
$$
M=M_U\cup M_D,
$$
where $M_U$ and $M_D$ denote respectively the upper and lower layers (these notions will be precisely defined below in Section \ref{Sec::3}). We define:
$$
X=\left\{
\mathfrak{n}\in M_U:
\mathfrak{n}\text{ is connected to some atom of }M_D
\right\}.
$$
Thus, $X$ consists of the upper-layer atoms that are connected to the lower layer. The analysis is based on the following four toy models that depend on the structure of $X$
\begin{itemize}
    \item \textbf{Toy model I}: no $2$-connections, in the simplest configuration.

    \item \textbf{Toy model I+}: no $2$-connections and $X$ has few connected components.

    \item \textbf{Toy model II}: $X$ has many connected components.

    \item \textbf{Toy model III}: there are many $2$-connections.
\end{itemize}
A $2$-connection is a bond joining two degree-$2$ atoms,
one in $M_U$ and one in $M_D$. 
The aim of the present article is to characterise the cutting algorithm
of~\cite{DHM25a} via a spanning tree. Our main theorem is the following.

\begin{theorem}
\label{main_theorem}
For all toy models, the cutting algorithm introduced in~\cite{DHM25a} builds a spanning tree
of the input molecule in a Kruskal manner.
\end{theorem}

The proof of Theorem~\ref{main_theorem} is based on an analysis of all the
steps of the cutting algorithm for all the toy models. At every step, we keep a maximum number of
the bonds removed by the cut, under the condition that the graph under
construction remains acyclic. The local statement used in the proof is
the following.

\begin{proposition}
\label{main_proposition}
We suppose given a step of the cutting algorithm described
in~\cite{DHM25a} on a molecule $M$ and an acyclic graph $G=(V,E)$, where
$V$ contains the atoms of $M$ and $E$ is a set of edges disjoint from the
bonds of $M$. Then one can add to $G$ a maximum number of edges removed by
the cutting step under consideration such that the new graph obtained is
still acyclic.
\end{proposition}

We prove Proposition~\ref{main_proposition} in Section~\ref{Sec::3} by reviewing
the different steps of the algorithms introduced in~\cite{DHM25a}. The
proof of Theorem~\ref{main_theorem} then follows from an induction on the
construction of the acyclic graph. When the cutting algorithm terminates,
this graph contains all the atoms and becomes a spanning tree of the
initial molecule.
We call this construction Kruskal-like because, as in Kruskal's
algorithm~\cite{K56}, one keeps edges as long as they do not create
a cycle, selecting them with respect to some weight priority.

The paper is organised as follows. In Section~\ref{Sec::2}, we recall the definitions
of molecules used in~\cite{DHM25a}. In Section~\ref{Sec::3}, we study the different
steps of the cutting algorithm and prove
Proposition~\ref{main_proposition}. We then prove
Theorem~\ref{main_theorem} and describe the spanning tree obtained from
the cutting procedure.

\subsection*{Acknowledgements}
{\small
Y.B. and V.C. gratefully acknowledge funding support from the European Research Council
(ERC) through the ERC Starting Grant Low Regularity Dynamics via Decorated Trees
(LoRDeT), grant agreement No. 101075208. Views and opinions expressed are however
those of the author(s) only and do not necessarily reflect those of the European Union or the
European Research Council. Neither the European Union nor the granting authority can be
held responsible for them. Y.B. also thanks the ”Institut des Hautes Etudes Scientifiques”
(IHES) for a long research stay from 7th of January to 21st of March 2025, where the main
idea for this work emerged. Y. B. gratefully acknowledges Thierry Bodineau for interesting
discussions on the Boltzmann kinetic equation and for suggesting to apply for a long stay at
IHES. Y. B. thanks José Bruned for interesting discussions on mechanical systems where
one uses similar graphical tools. Y. B. thanks Ismaël Bailleul for pointing out the use of
spanning trees in Quantum Field Theory.}

\section{Molecules and cumulants}
\label{Sec::2}
\subsection{Molecules}
\begin{definition}[Molecule]
A \emph{molecule} is a quadruple
$$
M=(\mathcal{M},\mathcal{E},\mathcal{P},\mathcal{L}),
$$
where:
\begin{enumerate}
\item $\mathcal{M}$ is a finite set of nodes, called \emph{atoms}. Each atom
$\mathfrak{n}\in\mathcal{M}$ is assigned
a layer $\ell[\mathfrak{n}]\in[\underline{\ell}:\ell]\cap\mathbb{N}$. We write
$$
M_{\ell'}
=
\{\mathfrak{n}\in\mathcal{M}:\ell[\mathfrak{n}]=\ell'\}.
$$

\item $\mathcal{E}$ is a finite set of \emph{edges}. Each edge is one of the following:
\begin{enumerate}
\item a \emph{bond}, joining two atoms $\mathfrak{n},\mathfrak{n}'\in\mathcal{M}$;
\item an \emph{end}, incident to exactly one atom $\mathfrak{n}\in\mathcal{M}$;
\item an \emph{empty end}, incident to no atom.
\end{enumerate}
Each edge incident to an atom is declared to be either a \emph{top edge} or a
\emph{bottom edge} at that atom. Each end is declared to be either \emph{free} or
\emph{fixed}; empty ends are, by convention, free and count as both top and bottom.

\item $\mathcal{P}$ is a collection of paths, called \emph{particle lines}. Each
particle line $\mathbf{p}$ starts from a bottom end at a layer $\ell_1(\mathbf{p})$, follows bonds through atoms by moving from
child to parent, and terminates at a top end at a layer $\ell_2(\mathbf{p})$. Every edge of $\mathcal{E}$ belongs to
exactly one particle line. In particular, every empty end is regarded as a particle
line. A particle line is called a root if it ends at a top edge.

\item $\mathcal{L}$ is a collection of \emph{initial links}, namely pairs
$(e,e')$ of bottom ends satisfying $\ell_1[e]=\ell_1[e']=\underline{\ell}$.

\end{enumerate}
The following admissibility conditions are imposed:
\begin{enumerate}
\item every atom has exactly two top edges and exactly two bottom edges;
\item the parent-child relation contains no directed cycle;
\item if $\mathfrak{n}$ is a parent of $\mathfrak{n}'$, then
$\ell[\mathfrak{n}]\geq \ell[\mathfrak{n}']$;
\item if an atom $\mathfrak{n}$ belongs to a particle line $p$, then
$\ell_1[p]\leq \ell[\mathfrak{n}]\leq \ell_2[p]$.
\end{enumerate}
Moreover, for all $\mathfrak{n}\in \mathcal{M}$, we denote $S_\mathfrak{n}\subset\mathcal{M}$ the set of successors of $\mathfrak{n}$ and $Z_\mathfrak{n}\subset\mathcal{M}$ the set of ancestors of $\mathfrak{n}$, for the order given by the edges. The degree of an atom is the number of its bonds plus free ends (not counting fixed ends).
\end{definition}
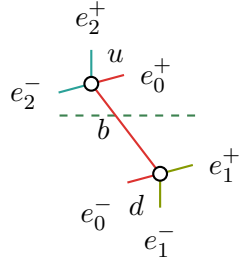
\begin{figure}[h]
\centering
\begin{tikzpicture}[scale=0.7]

\definecolor{pzero}{RGB}{220,50,47}
\definecolor{ptwo}{RGB}{133,153,0}
\definecolor{pthree}{RGB}{42,161,152}
\definecolor{greenline}{RGB}{60,130,80}

\def\dx{1.3}
\def\slope{0.35}
\def\gap{2.05}

\coordinate (d2) at (1*\dx,1*\slope);
\coordinate (u1) at (0*\dx,\gap + 0*\slope);

\coordinate (d2ll) at ($(d2)+(-0.618,-0.166)$);
\coordinate (d2ur) at ($(d2)+(0.618,0.166)$);
\coordinate (d2down) at ($(d2)+(0,-0.64)$);

\coordinate (u1ulow) at ($(u1)+(-0.618,-0.166)$);
\coordinate (u1ur) at ($(u1)+(0.618,0.166)$);
\coordinate (u1up) at ($(u1)+(0,0.64)$);

\draw[greenline, thick, dashed] (-0.6,1.45) -- (2.0,1.45);

\draw[pzero, thick] (d2ll) -- (d2);
\draw[pzero, thick] (d2) -- (u1);
\draw[pzero, thick] (u1) -- (u1ur);

\draw[ptwo, thick] (d2down) -- (d2);
\draw[ptwo, thick] (d2) -- (d2ur);

\draw[pthree, thick] (u1ulow) -- (u1);
\draw[pthree, thick] (u1) -- (u1up);

\foreach \p in {d2,u1}{
  \filldraw[white, draw=black, thick] (\p) circle (0.13);
}

\node[above right=4pt and 2pt] at (u1) {$u$};
\node[below left=4pt and 2pt] at (d2) {$d$};

\node[below left=2pt] at (d2ll) {$e_0^-$};
\node[below=2pt] at (d2down) {$e_1^-$};
\node[right=2pt] at (d2ur) {$e_1^+$};

\node[left=2pt] at ($(d2)!0.5!(u1)$) {$b$};

\node[left=2pt] at (u1ulow) {$e_2^-$};
\node[right=2pt] at (u1ur) {$e_0^+$};
\node[above=2pt] at (u1up) {$e_2^+$};

\end{tikzpicture}
\caption{Example of molecule with two atoms and two layers $M_U,M_D$, separated with a dashed line. Each color corresponds to a particle line. Moreover, we have $\mathcal{M}=\{d,u\}$, $\mathcal{E}=\{e_0^-,e_1^-,e_2^-,e_0^+,e_1^+,e_2^+,b\}$, $\mathcal{L}=\emptyset$.
}
\label{fig:moleculeexample}
\end{figure}

In \cite{DHM25a}, the authors make the distinction between $O$-atoms and $C$-atoms. The $O$-atoms are overlapped particles in the senses that the particles cross each other without collision. This is a mathematical artefact to conduct the proof. In the sequel, we work with collision history without these overlaps.
A molecule represent some collision scenario starting from an initial set of particles and letting it evolve following the hard sphere dynamics.
In the next section, we explain how molecules are related to analytic quantities.
\subsection{Cumulant formula}
We start from the ansatz \cite[Section 3]{DHM25a}:
$$f_s(\ell\tau,z_s)=\sum_{H\subset \{1,\ldots,s\}}\left(f^A(\ell\tau)\right)^{\otimes(\{1,\ldots,s\}\setminus H)}
E_H(\ell\tau,z_H)+\mathrm{Err}(\ell\tau,z_s)$$
and we will neglect the error term. We can control the cumulants using the molecules:
$$
\left\lvert E_H(\ell\tau,z_H)\right\rvert
\leq
\sum_{\underset{r(M)=H}{M\in\mathcal{F}}}
\left\lvert \mathcal{I}\mathcal{N}_{M}\right\rvert(z_H)
$$
where $r(M)$ is the set of particle lines penetrating the top horizontal line $\ell\tau$ and $\mathcal{F}$ is the set of molecules such that:
\begin{itemize}
    \item Every connected component of $M_{\ell'}$
    contains at least one particle line from
    $r(M_{\ell'})$.

    \item The number of atoms in $M_{\ell'}$ is bounded by
    $\lvert\log\varepsilon\rvert^{O(1)}$.

    \item The layer $M_{\ell'}$ is almost a forest.

    \item Every particle line in $r(M_{\ell'})$ is connected
    either to another particle line in
    $r(M_{\ell'})$, or to a particle line in
    $r(M_{\ell'-1})$.
\end{itemize}
As we want to control the $L^1$-norm of $\left\lvert \mathcal{I}\mathcal{N}_{M} \right\rvert\left(z_{r(M)}\right)$, we can establish the following expression for $\lVert\left\lvert \mathcal{I}\mathcal{N}_{M} \right\rvert
\left(z_{r(M)}\right)\rVert_{L^1}$ (\cite[Definition 3.23]{DHM25a}):
$$
\lVert\left\lvert \mathcal{I}\mathcal{N}_{M} \right\rvert
\left(z_{r(M)}\right)\rVert_{L^1}=\varepsilon^{(d-1)\lvert H\rvert}I_M(Q_M)$$
where $Q_M$ is a bounded function and :
$$
I_M(Q)=\varepsilon^{-(d-1)(\lvert\mathcal{E}\rvert-2\lvert \mathcal{M}\rvert)}\int_{\mathbb{R}^{2d\lvert\mathcal{E}\rvert}\times\mathbb{R}^{\lvert\mathcal{M}\rvert}}\left(\prod_{\mathfrak{n}\in\mathcal{M}}\Delta_\mathfrak{n}\right)Q\mathrm{d} z_\mathcal{E}\mathrm{d}t_\mathcal{M}
$$
with $\Delta_\mathfrak{n}$ is some distribution defined in ~\cite[Equation 2.2.2]{DHM25a}. Thanks to the product $\displaystyle\prod_{\mathfrak{n}\in\mathcal{M}}\Delta_\mathfrak{n}$, the operators $I$ verify the crucial identity: if a molecule $M$ is cut in two molecules $M_1,M_2$, we have $I_M=I_{M_1}\circ I_{M_2}$. This composition reflects the fact that cutting the molecule separates the integration along a particle line into two successive parts: the endpoint created by the cut is treated as a fixed end for one molecule and as a free end for the other, allowing the two corresponding integration operators to be applied successively. Hence, we can apply cutting algorithms to establish estimates on $I_M$.

Let us give an explicit example of computation of $I_M,Q_M$. We pick $M$ the molecule described in Figure \ref{fig:moleculeexample}. By  \cite[Definition 7.3]{DHM25a}, we have, as each $\Delta$ produces $5$ multiplicative terms:
$$
\begin{aligned}
I_M(Q)
&=
\varepsilon^{-3(d-1)}
\int_{\mathbb{R}^{14d}\times\mathbb{R}^2}
\Delta_d\Delta_u Q
\mathrm{d}z_{\mathcal{E}}
\mathrm{d}t_d\mathrm{d}t_u
\\
&=\varepsilon^{-3(d-1)}
\int_{\mathbb{R}^{14d}\times\mathbb{R}^2}
\delta\left(
x_b-x_{e_0^-}
+t_d(v_b-v_{e_0^-})
\right)\times
\delta\left(
x_{e_1^+}-x_{e_1^-}
+t_d(v_{e_1^+}-v_{e_1^-})
\right)
\\&\times
\delta\left(
\left\lvert
x_{e_0^-}-x_{e_1^-}
+t_d(v_{e_0^-}-v_{e_1^-})
\right\rvert
-\varepsilon
\right)
\times
\left[
(v_{e_0^-}-v_{e_1^-})\cdot\omega_d
\right]_{-}
\\
&\times
\delta\left(
v_b-v_{e_0^-}
+
\left[
(v_{e_0^-}-v_{e_1^-})\cdot\omega_d
\right]\omega_d
\right)
\times
\delta\left(
v_{e_1^+}-v_{e_1^-}
-
\left[
(v_{e_0^-}-v_{e_1^-})\cdot\omega_d
\right]\omega_d
\right)
\\
&\times
\delta\left(
x_{e_0^+}-x_b
+t_u(v_{e_0^+}-v_b)
\right)
\times
\delta\left(
x_{e_2^+}-x_{e_2^-}
+t_u(v_{e_2^+}-v_{e_2^-})
\right)
\\
&\times
\delta\left(
\left\lvert
x_b-x_{e_2^-}
+t_u(v_b-v_{e_2^-})
\right\rvert
-\varepsilon
\right)
\times
\left[
(v_b-v_{e_2^-})\cdot\omega_u
\right]_{-}
\\
&\times
\delta\left(
v_{e_0^+}-v_b
+
\left[
(v_b-v_{e_2^-})\cdot\omega_u
\right]\omega_u
\right)
\times
\delta\left(
v_{e_2^+}-v_{e_2^-}
-
\left[
(v_b-v_{e_2^-})\cdot\omega_u
\right]\omega_u
\right)
\\
&\times
Q(z_{\mathcal E},t_d,t_u)
\mathrm{d}z_{\mathcal E}
\mathrm{d}t_d
\mathrm{d}t_u.
\end{aligned}
$$
with $
\omega_d
=
\dfrac{1}{\varepsilon}\left(
x_{e_0^-}-x_{e_1^-}
\right),
\omega_u
=
\dfrac{1}{\varepsilon}\left(
x_b-x_{e_2^-}
+t_u(v_b-v_{e_2^-})
\right)$ ; and:
$$
\begin{aligned}
Q_M
&=
\mathbf{1}_{\mathcal L}^{\varepsilon}
\prod_{j=0}^{2}
f^A
\left(
(\ell_1[\mathbf{p}_j]-1)\tau,
x'_{\mathbf{p}_j}
+
(\ell_1[\mathbf{p}_j]-\underline{\ell})\tau v'_{\mathbf{p}_j},
v'_{\mathbf{p}_j}
\right)
\\
&=
f_0\left(x'_{\mathbf{p}_0},v'_{\mathbf{p}_0}\right)
f_0\left(x'_{\mathbf{p}_1},v'_{\mathbf{p}_1}\right)
f^A\left(
\tau,
x'_{\mathbf{p}_2}+\tau v'_{\mathbf{p}_2},
v'_{\mathbf{p}_2}
\right).
\end{aligned}
$$

\section{Cutting algorithm and toy models}
\label{Sec::3}
In ~\cite[Section 11.6]{DHM25a}, the authors show that it is possible to reduce the general reducing procedure of a multi-layer molecule to a reducing procedure for a $2$-layer molecule.
\begin{definition}[Two-layer molecule]
Let $M$ be a molecule with exactly two layers
$$
M=M_U\cup M_D,
$$
where $M_U=M_{\ell_U}$ is the upper molecule, $M_D=M_{\ell_D}$ is the lower molecule and $\ell_U>\ell_D$. We set
$$
X=\left\{
\mathfrak{n}\in M_U,
\mathfrak{n}\text{ is connected to some atom of } M_D
\right\}.
$$
A \emph{2-connection} is a bond joining two degree-$2$ atoms,
one in $M_U$ and one in $M_D$. We denote by $\#2-\mathrm{conn}$ the number of $2$-connections.
Furthermore, let $\#\mathrm{comp}(X)$ be the number of connected components of $X$.
\end{definition}
\begin{figure}[H]
\centering
\begin{tikzpicture}[scale=0.6]

\definecolor{greenline}{RGB}{60,130,80}
\definecolor{blueconn}{RGB}{0,90,255}
\definecolor{rednote}{RGB}{255,70,40}

\def\dx{1.15}
\def\slope{0.35}
\def\gap{2.05}

\coordinate (d1) at (0*\dx,0*\slope);
\coordinate (d2) at (1*\dx,1*\slope);
\coordinate (d3) at (2*\dx,2*\slope);
\coordinate (d4) at (3*\dx,3*\slope);
\coordinate (d5) at (4*\dx,4*\slope);

\coordinate (u1) at (0*\dx,\gap + 0*\slope);
\coordinate (u2) at (1*\dx,\gap + 1*\slope);
\coordinate (u3) at (2*\dx,\gap + 2*\slope);
\coordinate (u4) at (3*\dx,\gap + 3*\slope);
\coordinate (u5) at (4*\dx,\gap + 4*\slope);

\draw[black, thick] (d1) -- (d2) -- (d3) -- (d4) -- (d5);
\draw[black, thick] (u1) -- (u2) -- (u3) -- (u4) -- (u5);

\draw[black, thick] (u1) -- ($ (u1) + (-0.306,-0.093) $);
\draw[black, thick] (d1) -- ($ (d1) + (-0.306,-0.093) $);

\draw[black, thick] (u5) -- ($ (u5) + (0.306,0.093) $);
\draw[black, thick] (d5) -- ($ (d5) + (0.306,0.093) $);

\draw[greenline, thick, dashed] (-0.8,1.85) -- (6.8,1.85);

\foreach \p in {d1,d2,d3,d4,d5}{
  \draw[black, thick] (\p) -- ($(\p)+(0,-0.32)$);
}

\foreach \p in {u1,u2,u3,u4,u5}{
  \draw[black, thick] (\p) -- ($(\p)+(0,0.32)$);
}

\foreach \p in {d1,d2,d3,d4,d5}{
  \filldraw[white, draw=black, thick] (\p) circle (0.13);
}

\foreach \p in {u1,u2,u3,u4,u5}{
  \filldraw[white, draw=black, thick] (\p) circle (0.13);
}

\draw[blueconn, thick] (u1) -- (d1);
\draw[blueconn, thick] (u2) -- (d3);
\draw[blueconn, thick] (u3) -- (d2);
\draw[blueconn, thick] (u4) -- (d5);
\draw[blueconn, thick] (u5) -- (d4);

\node[above=5pt] at (u1) {$u_1$};
\node[above=5pt] at (u2) {$u_2$};
\node[above=5pt] at (u3) {$u_3$};
\node[above=5pt] at (u4) {$u_4$};
\node[above=5pt] at (u5) {$u_5$};

\node[below=6pt] at (d1) {$d_1$};
\node[below=6pt] at (d2) {$d_2$};
\node[below=6pt] at (d3) {$d_3$};
\node[below=6pt] at (d4) {$d_4$};
\node[below=6pt] at (d5) {$d_5$};

\node[right] at (6.2,2.9) {$M_U$};
\node[right] at (6.2,0.9) {$M_D$};

\end{tikzpicture}
\caption{Exemple of $2$-layer molecule. All the atoms of $M_U$ are in $X$, $\#2-\mathrm{conn}=5$ and $\mathrm{comp}(X)=1$.}
\end{figure}
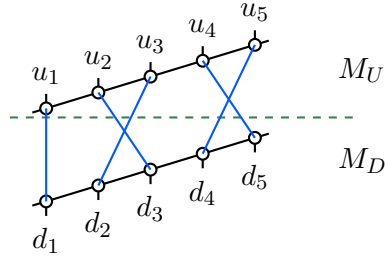

The analysis of two-layer molecules is based on a decomposition into
four extremal model situations.
The below cases form the basic decomposition used in the analysis of
two-layer molecules. Roughly speaking, either many connected components
are present (Toy model II), many $2$-connections are present (Toy model
III), or neither phenomenon occurs. The latter situation corresponds to
Toy model I+ and requires the full cutting algorithm. Toy model I is
a simplified version of this regime, introduced to isolate the main
combinatorial mechanism. 

\begin{definition}
We say that $M_D$ is \emph{proper} if the following conditions hold:
\begin{enumerate}
\item $M_D$ contains no atoms of degree $1$ or $2$;
\item no two degree-$3$ atoms are adjacent;
\item no degree-$4$ atom is adjacent to two degree-$3$ atoms.
\end{enumerate}
\end{definition}
\begin{theorem}\label{thm:Kruskal}
The algorithms exposed in ~\cite{DHM25a}, Definitions 11.4, 11.13, 11.17, 11.21, build a spanning tree of the toy model molecules $M$ in a Kruskal manner.
\end{theorem}
The proof of Theorem~\ref{thm:Kruskal} relies on an analysis of the various
steps of the algorithms described in ~\cite{DHM25a} by keeping at each step a maximum number of edges removed. These edges are part of an acyclic graph in construction. This procedure is given in the subsections below, following a Kruskal algorithm with this order, depending also on $(t_\mathfrak{n})_{\mathfrak{n}\in\mathcal{M}},(x_e,v_e)_{e\in\mathcal{E}}$ (for more details, see ~\cite[Definition 9.4]{DHM25a}):
\begin{table}[H]
\centering
\renewcommand{\arraystretch}{1.4}
\begin{tabular}{|c|p{0.65\textwidth}|p{0.2\textwidth}|}
\hline
\textbf{Type} & \textbf{Submolecules} & \textbf{Effect of integration} \\
\hline
Good
&
 $\{3\}$ with fixed end; 
 $\{33\}$ with two fixed ends on the $3$-atoms; 
 $\{44\}$; 
 $\{4\}$; empty end which is part of some pair of $\mathcal{L}$
&
Gain $\varepsilon^\nu$ for some $\nu>0$
\\
\hline
Normal
&
$\{2\}$; $\{3\}$ (except the good ones); $\{33A\}$ (except the good ones)
&
Neither gain nor loss of any power of $\varepsilon$
\\
\hline
Bad
&
$\{4\}$ and empty ends
&
Loss due to wasted dimensions of integration
\\
\hline
\end{tabular}
\caption{Classification of submolecules according to their contribution after integration.}
\label{tab:submolecule-classification}
\end{table}
\begin{remark}
 We can compare this algorithm with the algorithm for NLS described in \cite{DH23a} and \cite{BC26}. In the NLS case, the algorithm has only one long thread with a global Kruskal order, with a notion of good, normal and bad submolecules linked with the gain and loss in phase integrals estimates. The notion of good, normal and bad is the same in the Boltzmann case, but there is no more one thread: there are four toy models representing four extremal cases depending on the geometric configuration of the molecule. In addition, some toy models use a sub-algorithm, the UP algorithm, which appears to be a Prim-like algorithm.
 \end{remark}
\subsection{Toy model I}
\begin{definition}[Toy model I]
A two-layer molecule $M=M_U\cup M_D$ is said to be a
\emph{toy model I} if $X=M_U$ and $\#2-\mathrm{conn}=0$.
In other words, every atom of the upper layer is connected to the lower
layer and no $2$-connection is present.
\end{definition}

Toy model I is the simplest configuration and serves as a prototype for
the cutting algorithm.

\begin{proof}[of Proposition~\ref{main_proposition} for Toy model I]
We perform the proof of the previous proposition by reviewing
all the steps of the algorithm given in \cite[Definition 11.4]{DHM25a}.

Let $M$ be a toy model I. We define the following cutting sequence.

\begin{enumerate}
\item If not all atoms in $M_U$ have been cut, then choose a lowest atom
$\mathfrak{n}\in M_U$ that has not been cut
(i.e. either $\mathfrak{n}$ has no child in $M_U$ or all its children in $M_U$
have been cut).

\item If $\mathfrak{n}$ is adjacent to an atom $\mathfrak{m}\in M_D$ which has degree $3$, then cut $\{\mathfrak{n},\mathfrak{m}\}$ as free and cut $\mathfrak{m}$ as free from $\{\mathfrak{n},\mathfrak{m}\}$ if $\mathfrak{n}$ has degree $4$; otherwise just cut $\mathfrak{n}$ as free. In the first case, we add to the acyclic graph $G$ the atom $\mathfrak{m}$ and the maximum amount of bonds linked to $\mathfrak{m}$ such that there is no cycle created in $G$. In the second case, we add to the acyclic graph $G$ the atom $\mathfrak{n}$ and the maximum amount of bonds linked to $\mathfrak{n}$ such that there is no cycle created in $G$.
\begin{figure}[h]
\centering
\begin{tikzpicture}[scale=0.7]

\draw[greenline,thick,dashed] (-1.5,1.1) -- (1.8,1.1);

\node[cutatom] (n) at (0,1.75) {};
\node[cutatom] (m) at (0.4,0.55) {};

\draw[black,thick] (n) -- (-0.6,2.35);
\draw[black,thick] (n) -- (0.6,2.35);

\draw[rednote,very thick] (n) -- (m);
\draw[black,thick] (n) -- (-0.55,1.15);

\draw[blueconn,thick] (m) -- (-0.2,-0.05);
\draw[blueconn,thick] (m) -- (1.0,-0.05);

\node[left=4pt] at (n) {$\mathfrak{n}$};
\node[right=4pt] at (m) {$\mathfrak m$};

\node[right] at (1.35,1.7) {$M_U$};
\node[right] at (1.35,0.5) {$M_D$};

\end{tikzpicture}
\caption{Step 2: illustration of the first case}
\end{figure}
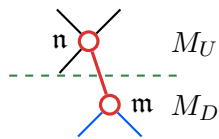

\item If $M_D$ contains two adjacent degree-$3$ atoms
(say $\mathfrak{r}$ and $\mathfrak{r}'$), then cut $\{\mathfrak{r},\mathfrak{r}'\}$ as free; repeat until $M_D$
has no adjacent degree-$3$ atoms. We add to the acyclic graph $G$ the atoms $\mathfrak{r},\mathfrak{r}'$ and the maximum amount of bonds linked to $\mathfrak{r},\mathfrak{r}'$ such that there is no cycle created in $G$.
\begin{figure}[H]
\centering
\begin{tikzpicture}[scale=0.7]

\draw[greenline,thick,dashed] (-1.7,1.3) -- (1.7,1.3);

\node[cutatom] (r) at (-0.55,0.55) {};
\node[cutatom] (rp) at (0.55,0.55) {};

\draw[rednote,very thick] (r) -- (rp);

\draw[blueconn,thick] (r) -- (-1.05,1.05);
\draw[blueconn,thick] (r) -- (-1.05,0.05);

\draw[blueconn,thick] (rp) -- (1.05,1.05);
\draw[blueconn,thick] (rp) -- (1.05,0.05);

\node[below=5pt] at (r) {$\mathfrak r$};
\node[below=5pt] at (rp) {$\mathfrak r'$};

\node[right] at (1.35,1.7) {$M_U$};
\node[right] at (1.35,0.5) {$M_D$};

\end{tikzpicture}
\caption{Step 3}
\end{figure}
\item If $M_D$ contains two degree-$3$ atoms $\mathfrak{r}$ and $\mathfrak{r}''$ both adjacent to the same degree-$4$ atom $\mathfrak{r}'$, then cut $\{\mathfrak{r},\mathfrak{r}',\mathfrak{r}''\}$ as free and cut $\mathfrak{r}$ as free from $\{\mathfrak{r},\mathfrak{r}',\mathfrak{r}''\}$; then go to (3)--(4) and repeat until $M_D$ becomes proper ( in \cite[Definition 11.3]{DHM25a}).
\begin{figure}[H]
\centering
\begin{tikzpicture}[scale=0.7]

\draw[greenline,thick,dashed] (-1.8,1.5) -- (1.8,1.5);

\node[cutatom] (r) at (-0.9,0.25) {};
\node[cutatom] (rp) at (0,0.85) {};
\node[cutatom] (rpp) at (0.9,0.25) {};

\draw[rednote,very thick] (r) -- (rp);
\draw[rednote,very thick] (rp) -- (rpp);

\draw[black,thick] (rp) -- (-0.45,1.4);
\draw[black,thick] (rp) -- (0.45,1.4);

\draw[black,thick] (r) -- (-1.45,0.75);
\draw[black,thick] (r) -- (-1.45,-0.25);

\draw[black,thick] (rpp) -- (1.45,0.75);
\draw[black,thick] (rpp) -- (1.45,-0.25);

\node[below=4pt] at (r) {$\mathfrak r$};
\node[below=4pt] at (rp) {$\mathfrak r'$};
\node[below=4pt] at (rpp) {$\mathfrak r''$};

\node[right] at (1.45,1.8) {$M_U$};
\node[right] at (1.45,0.4) {$M_D$};

\end{tikzpicture}
\caption{Step 4}
\end{figure}
\item Go to (1) and choose the next lowest atom
$\mathfrak{n}\in M_U$ that has not been cut, and so on. We add to the acyclic graph $G$ the atom $\mathfrak{n}$ and the maximum amount of bonds linked to $\mathfrak{n}$ such that there is no cycle created in $G$.

\item Suppose all atoms in $M_U$ have been cut. If $M_D$ is proper,
then cut any degree-$3$ atom; if not, then go to (3)--(4) and repeat
to remove them, and next cut any degree-$3$ atom, and so on. In any case, we add to the acyclic graph $G$ the cut atoms and the maximum amount of bonds linked to them such that there is no cycle created in $G$.
\end{enumerate}
\end{proof}
All these steps may cut good, normal or bad submolecules, as described in this table.
\begin{table}[H]
\centering
\renewcommand{\arraystretch}{1.3}
\begin{tabular}{|c|c|c|c|}
\hline
\textbf{Step} & \textbf{Cut} & \textbf{Location} & \textbf{Type} \\
\hline
(1) & Selection of the lowest atom in $M_U$ & $M_U$ & -- \\
\hline
(2) & $\{33\}$ & Across $M_U$ and $M_D$ & Good \\
    & $\{2\}$ or $\{3\}$ & $M_U$ & Normal \\
    & $\{4\}$ & $M_U$ & Bad \\
\hline
(3) & $\{33\}$ & $M_D$ & Good \\
\hline
(4) & $\{343\}\longrightarrow \{3\}+\{33\}$ & $M_D$ & Normal + Good \\
\hline
(5) & Selection of the next atom in $M_U$ & $M_U$ & -- \\
\hline
(6) & $\{3\}$ & $M_D$ & Normal \\
    & Cuts produced through steps (3)--(4) & $M_D$ & Good / Normal \\
\hline
\end{tabular}
\caption{Classification and location of the cuts produced by the toy model I algorithm.}
\end{table}
It is not obvious to read the Kruskal priority order in the previous table. We provide a second table below listing the priority.
\begin{table}[H]
\centering
\renewcommand{\arraystretch}{1.3}
\begin{tabular}{|c|c|c|c|}
\hline
\textbf{Priority} & \textbf{Cut} & \textbf{Location} & \textbf{Type} \\
\hline
(1) & $\{33\}$ & Across $M_U$ and $M_D$, or in $M_D$ & Good \\
\hline
(2) & $\{343\}\longrightarrow \{3\}+\{33\}$ & $M_D$ & Normal + Good \\
\hline
(3) & Lowest atom & $M_U$ & -- \\
\hline
(4) & $\{3\}$ & $M_U$ or $M_D$ & Normal \\
\hline
\end{tabular}
\caption{Kruskal priority order for the cuts produced by the toy model I algorithm.}
\label{tab:submolecule-priority}
\end{table}

Here is an example of type I molecule reduction, with the construction of the associated spanning tree.
\begin{figure}[H]
\centering
%
\caption{Initial type I molecule}
\end{figure}
In the following figures, we draw the evolutive molecule on the left and its associated Kruskal spanning tree on the right.
\begin{figure}[H]
\centering
%
\caption{Molecule after applying step $2$ on the $\{4\}$-atom $u_1$}
\end{figure}
\begin{figure}[H]
\centering
%
\caption{Molecule after applying in a new loop step $2$ on the $\{3\}$-atom $u_2$}
\end{figure}
\begin{figure}[H]
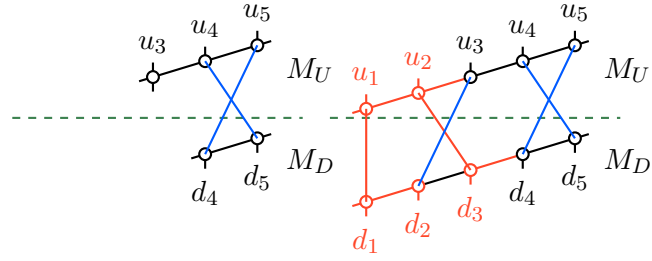

\centering
%
\caption{Molecule after applying step $4$ on the $\{343\}$-submolecule $d_1-d_2-d_3$}
\end{figure}
\begin{figure}[H]
\centering
%
\caption{Molecule after applying in a new loop step $2$ on the $\{2\}$-atom $u_3$}
\end{figure}
\begin{figure}[H]
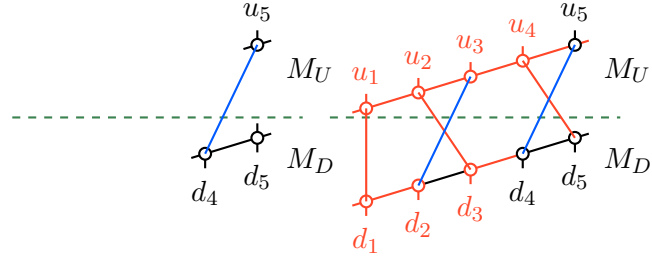

\centering
%
\caption{Molecule after applying in a new loop step $2$ on the $\{3\}$-atom $u_4$}
\end{figure}
\subsection{Toy model I+}
\begin{definition}[Toy model I+]
A two-layer molecule $M=M_U\cup M_D$ is said to be a
\emph{toy model I+} if $\#2-\mathrm{conn}=0$ and $\#\mathrm{comp}(X)\ll \lvert X\rvert$.
\end{definition}

In this regime the set $X$ has relatively few connected components,
hence large connected clusters are present.

\begin{proof}[of Proposition~\ref{main_proposition} for Toy model I+]
We perform the proof of the previous proposition by reviewing
all the steps of the algorithm given in ~\cite{DHM25a}.

Let $M$ be a toy model I+, we define the
following cutting sequence.

\begin{enumerate}
\item If $M_U$ has a degree-$2$ atom $\mathfrak{n}$, then cut it as free and add it to the acyclic graph $G$ with the maximum amount of bonds linked to $\mathfrak{n}$ such that there is no cycle created in $G$. Repeat until $M_U$ does not have degree-$2$ atoms.

\item Consider the set of all (remaining) atoms that either belong to
$X$, or have degree $3$; choose $\mathfrak{n}$ to be a lowest atom in this set
(if this set is empty then choose $\mathfrak{n}$ to be any lowest atom).
Let $S_\mathfrak{n}$ be the set of descendants of $\mathfrak{n}$ in $M_U$, and fix $\mathfrak{n}$ and $S_\mathfrak{n}$ until the end of (6).

\item Starting from $\mathfrak{n}$, each time choose a highest atom $\mathfrak{p}$ in $S_\mathfrak{n}$ that has not been cut. If $\mathfrak{p}$ is adjacent to an atom $\mathfrak{q}\in M_D$ which has degree $3$, then cut $\{\mathfrak{p},\mathfrak{q}\}$ as free and cut $\mathfrak{q}$ as free from $\{\mathfrak{p},\mathfrak{q}\}$ if $\mathfrak{p}$ has degree $4$; otherwise just cut
$\mathfrak{p}$ as free. In any cases, add the cut atoms to the acyclic graph $G$ with the maximum amount of bonds linked to them such that there is no cycle created in $G$.
\begin{figure}[H]
\centering
\begin{tikzpicture}[scale=0.7]

\draw[greenline,thick,dashed] (-1.8,1.5) -- (1.8,1.5);

\node[cutatom] (p) at (0,1.95) {};
\node[cutatom] (q) at (0.45,0.45) {};

\draw[rednote,very thick] (p) -- (q);

\draw[black,thick] (p) -- (-0.55,2.55);
\draw[black,thick] (p) -- (0.55,2.55);
\draw[black,thick] (p) -- (-0.45,1.35);

\draw[black,thick] (q) -- (-0.05,-0.15);
\draw[black,thick] (q) -- (0.95,-0.15);

\node[above=4pt] at (p) {$\mathfrak p$};
\node[below=4pt] at (q) {$\mathfrak q$};

\node[right] at (1.45,2.025) {$M_U$};
\node[right] at (1.45,0.625) {$M_D$};

\end{tikzpicture}
\caption{Step 3}
\end{figure}
\item If $M_D$ becomes non-proper, then perform the same steps (3)--(4) in the algorithm of the toy model I until it becomes proper again. Then go to (3) and repeat, until all atoms in $S_\mathfrak{n}$ have been cut.

\item When all atoms in $S_\mathfrak{n}$ have been cut, go to (1)--(2) and choose
the next $\mathfrak{n}$, and so on.

\item When all atoms in $M_U$ have been cut, perform the same step (6)
in Definition~11.4 to finish off $M_D$.
\end{enumerate}
\end{proof}
All these steps may cut good, normal or bad submolecules, as described in this table.
\begin{table}[H]
\centering
\renewcommand{\arraystretch}{1.3}

\resizebox{\textwidth}{!}{
\begin{tabular}{|c|c|c|c|}
\hline
\textbf{Step} & \textbf{Cut} & \textbf{Location} & \textbf{Type} \\
\hline
(1) & $\{2\}$ & $M_U$ & Normal \\
\hline
(2) & Selection of $\mathfrak{n}$ and $S_{\mathfrak{n}}$ & $M_U$ & -- \\
\hline
(3) & $\{33\}$ & Across $M_U$ and $M_D$ & Good \\
    & $\{2\}$ or $\{3\}$ & $M_U$ & Normal \\
    & $\{4\}$ & $M_U$ & Bad \\
    & $\{4,3\}\longrightarrow\{3\}+\{3\}$ 
    & Across $M_U$ and $M_D$ & Normal + Normal \\
\hline
(4) & $\{33\}$ & $M_D$ & Good \\
    & $\{343\}\longrightarrow\{3\}+\{33\}$ 
    & $M_D$ & Normal + Good \\
\hline
(5) & Selection of the next $\mathfrak{n}$ & $M_U$ & -- \\
\hline
(6) & $\{3\}$ & $M_D$ & Normal \\
    & Cuts produced through steps (3)--(4) of toy model I 
    & $M_D$ & Good / Normal \\
\hline
\end{tabular}
}

\caption{Classification and location of the cuts produced by the toy model I+ algorithm.}
\end{table}
The table of priority is the same as for toy model I given in Table \ref{tab:submolecule-priority}.
Here is an example of type I+ molecule reduction, with the construction of the associated spanning tree.
\begin{figure}[H]
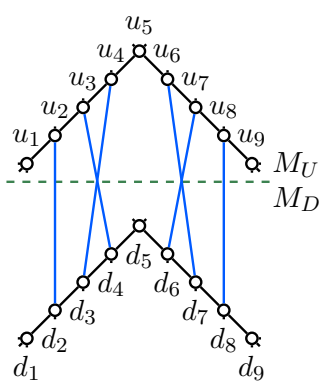

\centering
%
\caption{Initial type I+ molecule}
\end{figure}
As in toy model I, it is not obvious to read the Kruskal priority order in the previous table. The Kruskal order is the same than in toy model I.

In the following figures, we draw the evolutive molecule on the left and its associated Kruskal spanning tree on the right.
\begin{figure}[H]
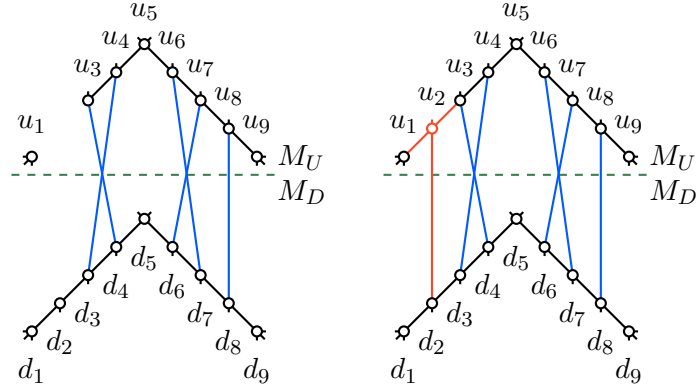

\centering
%
\caption{Molecule after applying step 3 on the $\{4\}$-atom in $X$ $u_2$ with $\mathfrak{n}=u_2$}
\end{figure}
\begin{figure}[H]
\centering

%
\caption{Molecule after applying step 4 on $\{3\}$-atom $u_1$ and after applying in a new loop step 3 on the $\{3\}$-atom in $X$ $u_3$ with $\mathfrak{n}=u_3$}
\end{figure}
\begin{figure}[H]
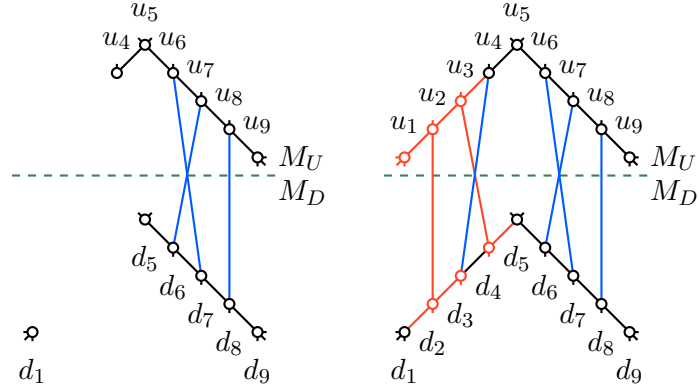

\centering
%
\caption{Molecule after applying step 4 to the $\{343\}$-submolecule $d_2-d_3-d_4$}
\end{figure}

\begin{figure}[H]
\centering
%
\caption{Molecule after applying step 1 on the $\{2\}$-atom $u_4$}
\end{figure}
\begin{figure}[H]
\centering
%
\caption{Molecule after applying step 3 on the $\{4\}$-atom in $X$ $u_8$ with $\mathfrak{n}=u_8$}
\end{figure}
\begin{figure}[H]
\centering
%
\caption{Molecule after applying step 4 on $\{3\}$-atom $u_9$ and after applying in a new loop step 3 on the $\{3\}$-atom in $X$ $u_7$ with $\mathfrak{n}=u_7$}
\end{figure}
\begin{figure}[H]
\centering
%
\caption{Molecule after applying step 4 to the $\{33\}$-submolecule $d_5-d_6$}
\end{figure}
\begin{figure}[H]
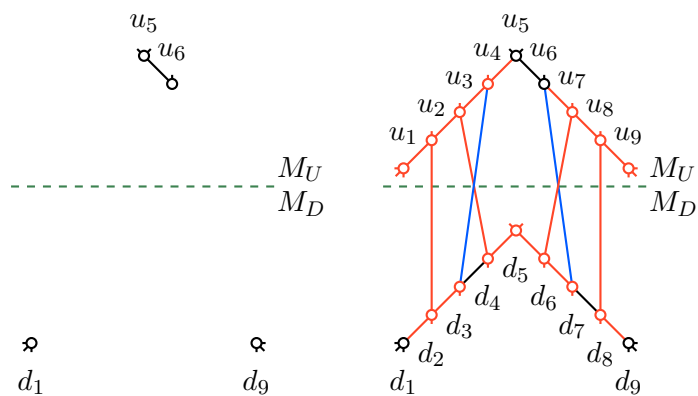

\centering
%
\caption{Molecule after applying step 4 to the $\{33\}$-submolecule $d_7-d_8$}
\end{figure}

\subsection{Toy model II}
\begin{definition}[Toy model II]
A two-layer molecule $M=M_U\cup M_D$ is said to be a
\emph{toy model II} if $\#2-\mathrm{conn}=0$ and $\#\mathrm{comp}(X)\gg\lvert X\rvert$
\end{definition}

Thus toy model II corresponds to the situation where the set $X$
contains many connected components.
We recall the UP algorithm described in  \cite[Definition 12.1]{DHM25a} in the Kruskal algorithm context.
\begin{definition}
Let $M$ be a molecule with no fixed top end. The algorithm UP is the following cutting sequence:
\begin{enumerate}

\item If $M$ contains a degree-$2$ atom $\mathfrak{n}$, then cut
$\mathfrak{n}$ as free and add it to the acyclic graph $G$ with the
maximum amount of bonds linked to $\mathfrak{n}$ such that there is no
cycle created in $G$. Repeat until $M$ has no degree-$2$ atoms.

\item Consider all remaining degree-$3$ atoms in $M$ and choose a lowest
one, denoted by $\mathfrak{n}$. If there is no degree-$3$ atom left,
choose $\mathfrak{n}$ to be any lowest atom of $M$. Fix $\mathfrak{n}$ and
$S_{\mathfrak{n}}$ until the end of (3).

\item Starting from $\mathfrak{n}$, each time choose a highest atom
$\mathfrak{m}$ in $S_{\mathfrak{n}}$ that has not been cut.
If $\mathfrak{m}$ has degree $3$ and has an parent
$\mathfrak{m}^{+}$ or an child $\mathfrak{m}^{-}$ which also has
degree $3$, then cut
$\{\mathfrak{m},\mathfrak{m}^{+}\}$ or
$\{\mathfrak{m},\mathfrak{m}^{-}\}$ as free.
Otherwise, cut $\mathfrak{m}$ as free.
In any case, add the cut atoms to the acyclic graph $G$ with the
maximum amount of bonds linked to them such that there is no cycle
created in $G$.

Repeat until all atoms in $S_{\mathfrak{n}}$ have been cut, then go back to (1).
\end{enumerate}
\end{definition}
Here is an example of molecule reduction using the UP algorithm.
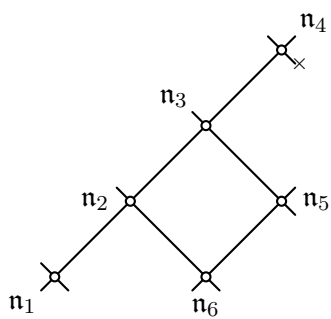
\begin{figure}[H]
\centering
\begin{tikzpicture}[
    line cap=round,
    line join=round,
    edge/.style={black,thick},
    vertex/.style={
        circle,
        draw=black,
        fill=white,
        inner sep=1.2pt,
        line width=0.8pt
    }
]

\coordinate (BL) at (-1,-1);
\coordinate (L)  at (0,0);
\coordinate (T)  at (1,1);
\coordinate (TR) at (2,2);
\coordinate (R)  at (2,0);
\coordinate (B)  at (1,-1);

\draw[edge] (L)--(T)--(R)--(B)--(L);
\draw[edge] (L)--(BL);
\draw[edge] (T)--(TR);

\draw[edge] (BL) -- ++(-0.18, 0.18);
\draw[edge] (BL) -- ++(-0.18,-0.18);
\draw[edge] (BL) -- ++( 0.18,-0.18);

\draw[edge] (L) -- ++(-0.18, 0.18);

\draw[edge] (T) -- ++(-0.18, 0.18);

\draw[edge] (TR) -- ++(-0.18, 0.18);
\draw[edge] (TR) -- ++( 0.18,-0.18);
\draw[edge] (TR) -- ++( 0.18, 0.18);

\draw[edge] (R) -- ++( 0.18,-0.18);
\draw[edge] (R) -- ++( 0.18, 0.18);

\draw[edge] (B) -- ++(-0.18,-0.18);
\draw[edge] (B) -- ++( 0.18,-0.18);
\
\node[font=\scriptsize] at (2.24,1.82) {$\times$};

\foreach \P in {BL,L,T,TR,R,B}{
    \node[vertex] at (\P) {};
}

\node[below left=4pt]  at (BL) {$\mathfrak{n}_1$};
\node[left=4pt]        at (L)  {$\mathfrak{n}_2$};
\node[above left=4pt]  at (T)  {$\mathfrak{n}_3$};
\node[above right=4pt] at (TR) {$\mathfrak{n}_4$};
\node[right=4pt]       at (R)  {$\mathfrak{n}_5$};
\node[below=4pt]       at (B)  {$\mathfrak{n}_6$};

\end{tikzpicture}
\caption{Initial molecule. The atom $\mathfrak{n}_4$ has one bottom fixed end which is marked via the symbol $ \times $.}
\end{figure}
\begin{figure}[H]
\centering
\begin{tikzpicture}[
    line cap=round,
    line join=round,
    edge/.style={black,thick},
    rededge/.style={rednote,thick},
    vertex/.style={
        circle,
        draw=black,
        fill=white,
        inner sep=1.2pt,
        line width=0.8pt
    },
    redvertex/.style={
        circle,
        draw=rednote,
        fill=white,
        inner sep=1.2pt,
        line width=0.8pt
    }
]

\begin{scope}[xshift=0cm]
\coordinate (BL) at (-1,-1);
\coordinate (L)  at (0,0);
\coordinate (T)  at (1,1);
\coordinate (R)  at (2,0);
\coordinate (B)  at (1,-1);

\draw[edge] (L)--(T)--(R)--(B)--(L);
\draw[edge] (L)--(BL);

\draw[edge] (BL) -- ++(-0.18, 0.18);
\draw[edge] (BL) -- ++(-0.18,-0.18);
\draw[edge] (BL) -- ++( 0.18,-0.18);

\draw[edge] (L) -- ++(-0.18, 0.18);

\draw[edge] (T) -- ++(-0.18, 0.18);
\draw[edge] (T) -- ++( 0.18, 0.18);

\draw[edge] (R) -- ++( 0.18,-0.18);
\draw[edge] (R) -- ++( 0.18, 0.18);

\draw[edge] (B) -- ++(-0.18,-0.18);
\draw[edge] (B) -- ++( 0.18,-0.18);

\foreach \P in {BL,L,T,R,B}{
    \node[vertex] at (\P) {};
}

\node[below left=4pt] at (BL) {$\mathfrak{n}_1$};
\node[left=4pt]       at (L)  {$\mathfrak{n}_2$};
\node[above left=4pt] at (T)  {$\mathfrak{n}_3$};
\node[right=4pt]      at (R)  {$\mathfrak{n}_5$};
\node[below=4pt]      at (B)  {$\mathfrak{n}_6$};
\end{scope}

\begin{scope}[xshift=6.8cm]
\coordinate (BL) at (-1,-1);
\coordinate (L)  at (0,0);
\coordinate (T)  at (1,1);
\coordinate (TR) at (2,2);
\coordinate (R)  at (2,0);
\coordinate (B)  at (1,-1);

\draw[edge]    (L)--(T)--(R)--(B)--(L);
\draw[edge]    (L)--(BL);
\draw[rededge] (T)--(TR);

\draw[edge] (BL) -- ++(-0.18, 0.18);
\draw[edge] (BL) -- ++(-0.18,-0.18);
\draw[edge] (BL) -- ++( 0.18,-0.18);

\draw[edge] (L) -- ++(-0.18, 0.18);

\draw[edge] (T) -- ++(-0.18, 0.18);

\draw[rededge] (TR) -- ++(-0.18, 0.18);
\draw[rededge] (TR) -- ++( 0.18,-0.18);
\draw[rededge] (TR) -- ++( 0.18, 0.18);
\node[font=\scriptsize,text=rednote] at (2.24,1.82) {$\times$};

\draw[edge] (R) -- ++( 0.18,-0.18);
\draw[edge] (R) -- ++( 0.18, 0.18);

\draw[edge] (B) -- ++(-0.18,-0.18);
\draw[edge] (B) -- ++( 0.18,-0.18);

\node[vertex]    at (BL) {};
\node[vertex]    at (L)  {};
\node[vertex]    at (T)  {};
\node[redvertex] at (TR) {};
\node[vertex]    at (R)  {};
\node[vertex]    at (B)  {};

\node[below left=4pt]  at (BL) {$\mathfrak{n}_1$};
\node[left=4pt]        at (L)  {$\mathfrak{n}_2$};
\node[above left=4pt]  at (T)  {$\mathfrak{n}_3$};
\node[above right=4pt] at (TR) {$\mathfrak{n}_4$};
\node[right=4pt]       at (R)  {$\mathfrak{n}_5$};
\node[below=4pt]       at (B)  {$\mathfrak{n}_6$};
\end{scope}

\end{tikzpicture}
\caption{UP algorithm with $\mathfrak{n}=\mathfrak{n}_4$ and $S_{\mathfrak{n}}=\{\mathfrak{n}_4,\mathfrak{n}_3,\mathfrak{n}_2,\mathfrak{n}_5,\mathfrak{n}_1,\mathfrak{n}_6\}$, cutting $\mathfrak{m}=\mathfrak{n}_4$.}
\end{figure}
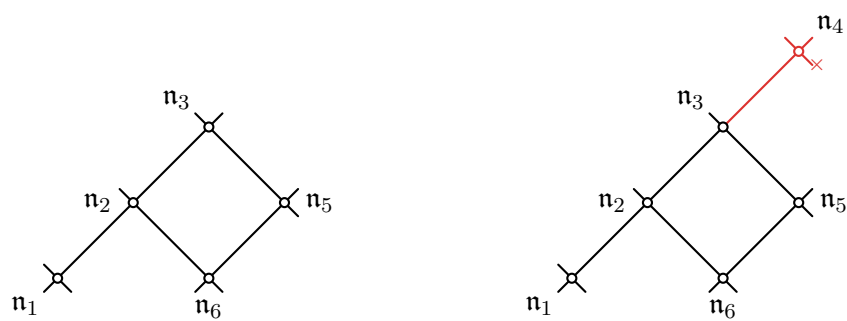

\begin{figure}[H]
\centering
\begin{tikzpicture}[
    line cap=round,
    line join=round,
    edge/.style={black,thick},
    rededge/.style={rednote,thick},
    vertex/.style={
        circle,
        draw=black,
        fill=white,
        inner sep=1.2pt,
        line width=0.8pt
    },
    redvertex/.style={
        circle,
        draw=rednote,
        fill=white,
        inner sep=1.2pt,
        line width=0.8pt
    }
]

\begin{scope}[xshift=0cm]
\coordinate (BL) at (-1,-1);
\coordinate (L)  at (0,0);
\coordinate (R)  at (2,0);
\coordinate (B)  at (1,-1);

\draw[edge] (L)--(B)--(R);
\draw[edge] (L)--(BL);

\draw[edge] (BL) -- ++(-0.18, 0.18);
\draw[edge] (BL) -- ++(-0.18,-0.18);
\draw[edge] (BL) -- ++( 0.18,-0.18);

\draw[edge] (L) -- ++(-0.18, 0.18);
\draw[edge] (L) -- ++( 0.18, 0.18);

\draw[edge] (R) -- ++(-0.18, 0.18);
\draw[edge] (R) -- ++( 0.18,-0.18);
\draw[edge] (R) -- ++( 0.18, 0.18);

\draw[edge] (B) -- ++(-0.18,-0.18);
\draw[edge] (B) -- ++( 0.18,-0.18);

\foreach \P in {BL,L,R,B}{
    \node[vertex] at (\P) {};
}

\node[below left=4pt] at (BL) {$\mathfrak{n}_1$};
\node[left=4pt]       at (L)  {$\mathfrak{n}_2$};
\node[right=4pt]      at (R)  {$\mathfrak{n}_5$};
\node[below=4pt]      at (B)  {$\mathfrak{n}_6$};
\end{scope}

\begin{scope}[xshift=6.8cm]
\coordinate (BL) at (-1,-1);
\coordinate (L)  at (0,0);
\coordinate (T)  at (1,1);
\coordinate (TR) at (2,2);
\coordinate (R)  at (2,0);
\coordinate (B)  at (1,-1);

\draw[edge]    (L)--(BL);
\draw[edge]    (B)--(L);
\draw[edge]    (R)--(B);
\draw[rededge] (L)--(T)--(R);
\draw[rededge] (T)--(TR);

\draw[edge] (BL) -- ++(-0.18, 0.18);
\draw[edge] (BL) -- ++(-0.18,-0.18);
\draw[edge] (BL) -- ++( 0.18,-0.18);

\draw[edge] (L) -- ++(-0.18, 0.18);

\draw[rededge] (T) -- ++(-0.18, 0.18);

\draw[rededge] (TR) -- ++(-0.18, 0.18);
\draw[rededge] (TR) -- ++( 0.18,-0.18);
\draw[rededge] (TR) -- ++( 0.18, 0.18);
\node[font=\scriptsize,text=rednote] at (2.24,1.82) {$\times$};

\draw[edge] (R) -- ++( 0.18,-0.18);
\draw[edge] (R) -- ++( 0.18, 0.18);

\draw[edge] (B) -- ++(-0.18,-0.18);
\draw[edge] (B) -- ++( 0.18,-0.18);

\node[vertex]    at (BL) {};
\node[vertex]    at (L)  {};
\node[redvertex] at (T)  {};
\node[redvertex] at (TR) {};
\node[vertex]    at (R)  {};
\node[vertex]    at (B)  {};

\node[below left=4pt]  at (BL) {$\mathfrak{n}_1$};
\node[left=4pt]        at (L)  {$\mathfrak{n}_2$};
\node[above left=4pt]  at (T)  {$\mathfrak{n}_3$};
\node[above right=4pt] at (TR) {$\mathfrak{n}_4$};
\node[right=4pt]       at (R)  {$\mathfrak{n}_5$};
\node[below=4pt]       at (B)  {$\mathfrak{n}_6$};
\end{scope}

\end{tikzpicture}
\caption{UP algorithm with the same fixed $\mathfrak{n}=\mathfrak{n}_4$ and $S_{\mathfrak{n}}$, cutting the highest remaining atom $\mathfrak{m}=\mathfrak{n}_3$.}
\end{figure}
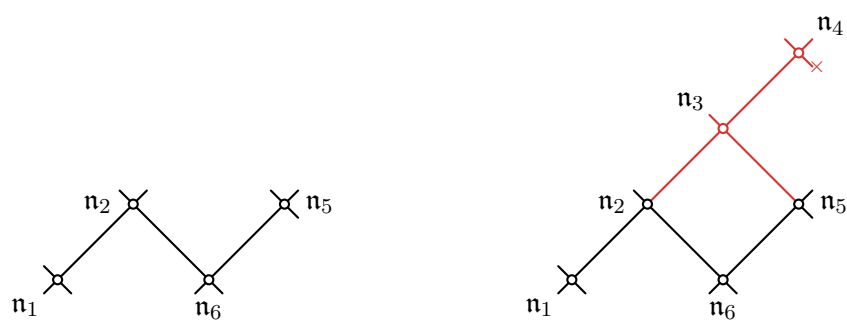

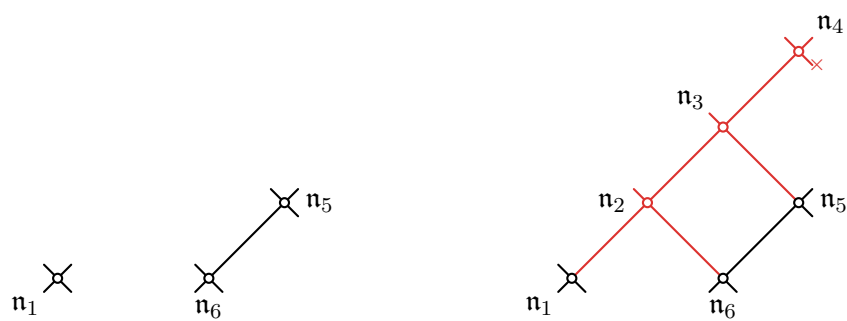
\begin{figure}[H]
\centering
\begin{tikzpicture}[
    line cap=round,
    line join=round,
    edge/.style={black,thick},
    rededge/.style={rednote,thick},
    vertex/.style={
        circle,
        draw=black,
        fill=white,
        inner sep=1.2pt,
        line width=0.8pt
    },
    redvertex/.style={
        circle,
        draw=rednote,
        fill=white,
        inner sep=1.2pt,
        line width=0.8pt
    }
]

\begin{scope}[xshift=0cm]
\coordinate (BL) at (-1,-1);
\coordinate (R)  at (2,0);
\coordinate (B)  at (1,-1);

\draw[edge] (R)--(B);

\draw[edge] (BL) -- ++(-0.18, 0.18);
\draw[edge] (BL) -- ++(-0.18,-0.18);
\draw[edge] (BL) -- ++( 0.18,-0.18);
\draw[edge] (BL) -- ++( 0.18, 0.18);

\draw[edge] (R) -- ++(-0.18, 0.18);
\draw[edge] (R) -- ++( 0.18,-0.18);
\draw[edge] (R) -- ++( 0.18, 0.18);

\draw[edge] (B) -- ++(-0.18, 0.18);
\draw[edge] (B) -- ++(-0.18,-0.18);
\draw[edge] (B) -- ++( 0.18,-0.18);

\foreach \P in {BL,R,B}{
    \node[vertex] at (\P) {};
}

\node[below left=4pt] at (BL) {$\mathfrak{n}_1$};
\node[right=4pt]      at (R)  {$\mathfrak{n}_5$};
\node[below=4pt]      at (B)  {$\mathfrak{n}_6$};
\end{scope}

\begin{scope}[xshift=6.8cm]
\coordinate (BL) at (-1,-1);
\coordinate (L)  at (0,0);
\coordinate (T)  at (1,1);
\coordinate (TR) at (2,2);
\coordinate (R)  at (2,0);
\coordinate (B)  at (1,-1);

\draw[rededge] (L)--(BL);
\draw[rededge] (B)--(L);
\draw[rededge] (L)--(T)--(R);
\draw[rededge] (T)--(TR);
\draw[edge]    (R)--(B);

\draw[edge] (BL) -- ++(-0.18, 0.18);
\draw[edge] (BL) -- ++(-0.18,-0.18);
\draw[edge] (BL) -- ++( 0.18,-0.18);

\draw[rededge] (L) -- ++(-0.18, 0.18);

\draw[rededge] (T) -- ++(-0.18, 0.18);

\draw[rededge] (TR) -- ++(-0.18, 0.18);
\draw[rededge] (TR) -- ++( 0.18,-0.18);
\draw[rededge] (TR) -- ++( 0.18, 0.18);
\node[font=\scriptsize,text=rednote] at (2.24,1.82) {$\times$};

\draw[edge] (R) -- ++( 0.18,-0.18);
\draw[edge] (R) -- ++( 0.18, 0.18);

\draw[edge] (B) -- ++(-0.18,-0.18);
\draw[edge] (B) -- ++( 0.18,-0.18);

\node[vertex]    at (BL) {};
\node[redvertex] at (L)  {};
\node[redvertex] at (T)  {};
\node[redvertex] at (TR) {};
\node[vertex]    at (R)  {};
\node[vertex]    at (B)  {};

\node[below left=4pt]  at (BL) {$\mathfrak{n}_1$};
\node[left=4pt]        at (L)  {$\mathfrak{n}_2$};
\node[above left=4pt]  at (T)  {$\mathfrak{n}_3$};
\node[above right=4pt] at (TR) {$\mathfrak{n}_4$};
\node[right=4pt]       at (R)  {$\mathfrak{n}_5$};
\node[below=4pt]       at (B)  {$\mathfrak{n}_6$};
\end{scope}

\end{tikzpicture}
\caption{UP algorithm with the same fixed $\mathfrak{n}=\mathfrak{n}_4$ and $S_{\mathfrak{n}}$. The highest remaining atoms are $\mathfrak{n}_2$ and $\mathfrak{n}_5$; resolving the tie to the left yields $\mathfrak{m}=\mathfrak{n}_2$.}
\end{figure}

\begin{figure}[H]
\centering
\begin{tikzpicture}[
    line cap=round,
    line join=round,
    edge/.style={black,thick},
    rededge/.style={rednote,thick},
    vertex/.style={
        circle,
        draw=black,
        fill=white,
        inner sep=1.2pt,
        line width=0.8pt
    },
    redvertex/.style={
        circle,
        draw=rednote,
        fill=white,
        inner sep=1.2pt,
        line width=0.8pt
    }
]

\begin{scope}[xshift=0cm]
\coordinate (BL) at (-1,-1);

\draw[edge] (BL) -- ++(-0.18, 0.18);
\draw[edge] (BL) -- ++(-0.18,-0.18);
\draw[edge] (BL) -- ++( 0.18,-0.18);
\draw[edge] (BL) -- ++( 0.18, 0.18);

\node[vertex] at (BL) {};

\node[below left=4pt] at (BL) {$\mathfrak{n}_1$};
\end{scope}

\begin{scope}[xshift=6.8cm]
\coordinate (BL) at (-1,-1);
\coordinate (L)  at (0,0);
\coordinate (T)  at (1,1);
\coordinate (TR) at (2,2);
\coordinate (R)  at (2,0);
\coordinate (B)  at (1,-1);

\draw[rededge] (L)--(BL);
\draw[rededge] (B)--(L);
\draw[rededge] (L)--(T)--(R)--(B);
\draw[rededge] (T)--(TR);

\draw[edge] (BL) -- ++(-0.18, 0.18);
\draw[edge] (BL) -- ++(-0.18,-0.18);
\draw[edge] (BL) -- ++( 0.18,-0.18);

\draw[rededge] (L) -- ++(-0.18, 0.18);

\draw[rededge] (T) -- ++(-0.18, 0.18);

\draw[rededge] (TR) -- ++(-0.18, 0.18);
\draw[rededge] (TR) -- ++( 0.18,-0.18);
\draw[rededge] (TR) -- ++( 0.18, 0.18);
\node[font=\scriptsize,text=rednote] at (2.24,1.82) {$\times$};

\draw[rededge] (R) -- ++( 0.18,-0.18);
\draw[rededge] (R) -- ++( 0.18, 0.18);

\draw[rededge] (B) -- ++(-0.18,-0.18);
\draw[rededge] (B) -- ++( 0.18,-0.18);

\node[vertex]    at (BL) {};
\node[redvertex] at (L)  {};
\node[redvertex] at (T)  {};
\node[redvertex] at (TR) {};
\node[redvertex] at (R)  {};
\node[redvertex] at (B)  {};

\node[below left=4pt]  at (BL) {$\mathfrak{n}_1$};
\node[left=4pt]        at (L)  {$\mathfrak{n}_2$};
\node[above left=4pt]  at (T)  {$\mathfrak{n}_3$};
\node[above right=4pt] at (TR) {$\mathfrak{n}_4$};
\node[right=4pt]       at (R)  {$\mathfrak{n}_5$};
\node[below=4pt]       at (B)  {$\mathfrak{n}_6$};
\end{scope}

\end{tikzpicture}
\caption{UP algorithm with the same fixed $\mathfrak{n}=\mathfrak{n}_4$ and $S_{\mathfrak{n}}$. Since $\deg(\mathfrak{n}_5)=\deg(\mathfrak{n}_6)=3$ and $\mathfrak{n}_6$ is a child of $\mathfrak{n}_5$, the pair $\{\mathfrak{n}_5,\mathfrak{n}_6\}$ is cut as a $\{33\}$ molecule.}
\end{figure}
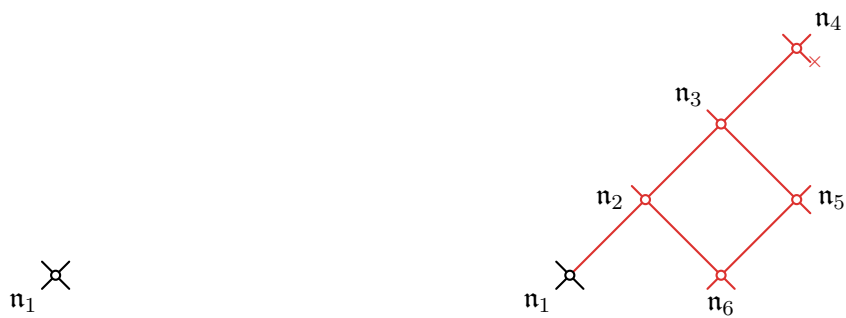

\begin{figure}[H]
\centering
\begin{tikzpicture}[
    line cap=round,
    line join=round,
    edge/.style={black,thick},
    rededge/.style={rednote,thick},
    vertex/.style={
        circle,
        draw=black,
        fill=white,
        inner sep=1.2pt,
        line width=0.8pt
    },
    redvertex/.style={
        circle,
        draw=rednote,
        fill=white,
        inner sep=1.2pt,
        line width=0.8pt
    }
]

\begin{scope}[xshift=0cm]
\node at (0.5,0.2) {};
\end{scope}

\begin{scope}[xshift=6.8cm]
\coordinate (BL) at (-1,-1);
\coordinate (L)  at (0,0);
\coordinate (T)  at (1,1);
\coordinate (TR) at (2,2);
\coordinate (R)  at (2,0);
\coordinate (B)  at (1,-1);

\draw[rededge] (L)--(T)--(R)--(B)--(L);
\draw[rededge] (L)--(BL);
\draw[rededge] (T)--(TR);

\draw[rededge] (BL) -- ++(-0.18, 0.18);
\draw[rededge] (BL) -- ++(-0.18,-0.18);
\draw[rededge] (BL) -- ++( 0.18,-0.18);
\draw[rededge] (BL) -- ++( 0.18, 0.18);

\draw[rededge] (L) -- ++(-0.18, 0.18);

\draw[rededge] (T) -- ++(-0.18, 0.18);

\draw[rededge] (TR) -- ++(-0.18, 0.18);
\draw[rededge] (TR) -- ++( 0.18,-0.18);
\draw[rededge] (TR) -- ++( 0.18, 0.18);
\node[font=\scriptsize,text=rednote] at (2.24,1.82) {$\times$};

\draw[rededge] (R) -- ++( 0.18,-0.18);
\draw[rededge] (R) -- ++( 0.18, 0.18);

\draw[rededge] (B) -- ++(-0.18,-0.18);
\draw[rededge] (B) -- ++( 0.18,-0.18);

\foreach \P in {BL,L,T,TR,R,B}{
    \node[redvertex] at (\P) {};
}

\node[below left=4pt]  at (BL) {$\mathfrak{n}_1$};
\node[left=4pt]        at (L)  {$\mathfrak{n}_2$};
\node[above left=4pt]  at (T)  {$\mathfrak{n}_3$};
\node[above right=4pt] at (TR) {$\mathfrak{n}_4$};
\node[right=4pt]       at (R)  {$\mathfrak{n}_5$};
\node[below=4pt]       at (B)  {$\mathfrak{n}_6$};
\end{scope}

\end{tikzpicture}
\caption{UP algorithm with the same fixed $n=\mathfrak{n}_4$ and $S_{\mathfrak{n}}$, cutting the last remaining atom $\mathfrak{m}=\mathfrak{n}_1$. The set $S_{\mathfrak{n}}$ is then exhausted.}
\end{figure}
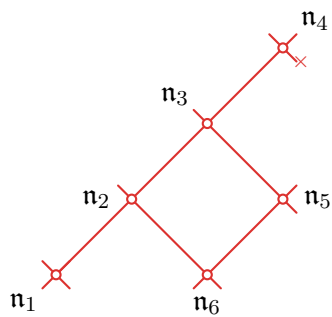

There is no optimisation in the UP algorithm, as it acts like a breadth-first search starting from a 3-degree atom. As we have to choose arbitrary the 3-degree root, the algorithm which builds the spanning tree of the molecule is a Prim-style algorithm, and not a Kruskal-style algorithm.

\begin{proof}[of Proposition~\ref{main_proposition} for Toy model II]
We perform the proof by reviewing
all the steps of the 3COMPUP algorithm given in \cite[Definition 11.17]{DHM25a}, .

Let $M$ be a molecule that is a tree, contains only degree-$3$ and
degree-$4$ atoms and has no top fixed end.
Let the set of degree-$3$ atoms in $M$ be $M_3$ with $\lvert M_3\rvert\geqslant 2$.

We define the following cutting sequence after cutting $M_D$, after applying the DOWN algorithm on $M_D$, which is the UP algorithm defined below, applied to the upside-down molecule).

\begin{enumerate}
\item Apply\cite[ Lemma~11.16]{DHM25a} to $M$ and select a connected no-bottom subset $A$ such that every child of atoms in $A$ is in $A$, and $A$ contains exactly two degree 3-atoms. This set is fixed until the end of (5).

\item Choose a lowest degree-$3$ atom $\mathfrak{n}$ in $A$ that has not been cut. Let $S_\mathfrak{n}$ be the set of descendants of $\mathfrak{n}$. This $\mathfrak{n}$ and $S_\mathfrak{n}$ are fixed until the end of (3).

\item Starting from $\mathfrak{n}$, each time choose a highest atom $\mathfrak{m}$ in $S_\mathfrak{n}$ that has not been cut. If $\mathfrak{m}$ has degree $3$ and has a parent $\mathfrak{m}^+$ or child $\mathfrak{m}^-$ that also has degree $3$, then cut $\{\mathfrak{m},\mathfrak{m}^\pm\}$ as free; otherwise cut $\mathfrak{m}$ as free. In any cases, add the cut atoms to the acyclic graph $G$ with the maximum amount of bonds linked to them such that there is no cycle created in $G$. Repeat until all atoms in $S_\mathfrak{n}$ have been cut.
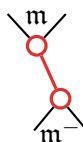
\begin{figure}[H]
\centering
\begin{tikzpicture}[scale=0.7]

\node[cutatom] (m) at (0,1.75) {};
\node[cutatom] (mp) at (0.45,0.75) {};

\draw[rednote,very thick] (m) -- (mp);

\draw[black,thick] (m) -- (-0.55,2.35);
\draw[black,thick] (m) -- (0.55,2.35);

\draw[black,thick] (mp) -- (-0.05,0.15);
\draw[black,thick] (mp) -- (0.95,0.15);

\node[above=4pt] at (m) {$\mathfrak{m}$};
\node[below=4pt] at (mp) {$\mathfrak{m}^-$};
\end{tikzpicture}
\caption{Step 3 (illustration of cutting operation only)}
\end{figure}
\item If $M$ has a degree-$2$ atom $\mathfrak{n}$, then cut it and add it to the acyclic graph $G$ with the maximum amount of bonds linked to $\mathfrak{n}$ such that there is no cycle created in $G$; repeat until $M$ has no degree-$2$ atoms.
Then go to (2) and choose the next $\mathfrak{n}$, and so on.

\item When all atoms in $A$ have been cut, remove $A$ from $M$ and add the atoms of $A$ to $G$ with the maximum amount of bound linked to them such that there is no cycle created in $G$. Repeat (1)--(4) on each remaining component until all components contain at most one degree-$3$ atom. Then finish by cutting the remaining atoms using the UP algorithm described below.
\end{enumerate}
\end{proof}
All the steps may cut good, normal or bad submolecules, as described in this table.
\begin{table}[H]
\centering
\renewcommand{\arraystretch}{1.3}

\resizebox{\textwidth}{!}{
\begin{tabular}{|c|c|c|c|}
\hline
\textbf{Step} & \textbf{Cut} & \textbf{Location} & \textbf{Type} \\
\hline

Initial step
& $\{33\}$ produced by DOWN
& $M_D$
& Good \\
& $\{2\}$ or $\{3\}$ produced by DOWN
& $M_D$
& Normal \\
& $\{4\}$ produced by DOWN
& $M_D$
& Bad \\
\hline

(1)
& Selection of the connected no-bottom set $A$
& $M_U$
& -- \\
\hline

(2)
& Selection of $\mathfrak{n}$ and $S_{\mathfrak{n}}$
& $M_U$
& -- \\
\hline

(3)
& $\{33\}$
& $M_U$
& Good \\
& $\{3\}$
& $M_U$
& Normal \\
\hline

(4)
& $\{2\}$
& $M_U$
& Normal \\
\hline

(5)
& $\{33\}$ produced while repeating (1)--(4)
& $M_U$
& Good \\
& $\{2\}$ or $\{3\}$ produced while repeating (1)--(4)
& $M_U$
& Normal \\
& $\{3\}$ produced by the final UP algorithm
& $M_U$
& Normal \\
\hline

\end{tabular}
}

\caption{Classification and location of the cuts produced by the toy model II algorithm.}
\end{table}
The main idea is to apply the UP algorithm to each submolecule $A$. There is no priority order in this algorithm.
Here is an example of type II molecule reduction, with the construction of the associated spanning tree.
\begin{figure}[H]
\centering

\begin{tikzpicture}[
    line cap=round,
    line join=round,
    edge/.style={black,thick},
    connection/.style={blueconn,thick},
    vertex/.style={
        circle,
        draw=black,
        fill=white,
        inner sep=1.2pt,
        line width=0.8pt
    }
]

\definecolor{greenline}{RGB}{60,130,80}
\definecolor{blueconn}{RGB}{0,90,255}

\def\dx{0.95}

\foreach \i in {1,...,9}{
    \pgfmathsetmacro{\x}{(\i-1)*\dx}
    \ifodd\i
        \coordinate (U\i) at (\x,1.35);
    \else
        \coordinate (U\i) at (\x,1.85);
    \fi
}

\foreach \i in {1,...,9}{
    \pgfmathsetmacro{\x}{(\i-1)*\dx}
    \ifodd\i
        \coordinate (D\i) at (\x,0.55);
    \else
        \coordinate (D\i) at (\x,0.05);
    \fi
}

\draw[edge]
    (U1)--(U2)--(U3)--(U4)--(U5)--(U6)--(U7)--(U8)--(U9);

\draw[edge]
    (D1)--(D2)--(D3)--(D4)--(D5)--(D6)--(D7)--(D8)--(D9);

\draw[edge] (U1) -- ++(-0.18, 0.18);
\draw[edge] (U1) -- ++( 0.18,-0.18);

\draw[edge] (U2) -- ++(-0.18, 0.18);
\draw[edge] (U2) -- ++( 0.18, 0.18);

\draw[edge] (U3) -- ++(-0.18,-0.18);
\draw[edge] (U3) -- ++( 0.18,-0.18);

\draw[edge] (U4) -- ++(-0.18, 0.18);
\draw[edge] (U4) -- ++( 0.18, 0.18);

\draw[edge] (U5) -- ++(0.18,-0.18);

\draw[edge] (U6) -- ++(-0.18, 0.18);
\draw[edge] (U6) -- ++( 0.18, 0.18);

\draw[edge] (U7) -- ++(-0.18,-0.18);
\draw[edge] (U7) -- ++( 0.18,-0.18);

\draw[edge] (U8) -- ++(-0.18, 0.18);
\draw[edge] (U8) -- ++( 0.18, 0.18);

\draw[edge] (U9) -- ++(0.18, 0.18);
\draw[edge] (U9) -- ++(0.18,-0.18);

\draw[edge] (D1) -- ++(-0.18, 0.18);
\draw[edge] (D1) -- ++(-0.18,-0.18);

\draw[edge] (D2) -- ++(-0.18,-0.18);
\draw[edge] (D2) -- ++( 0.18,-0.18);

\draw[edge] (D3) -- ++(-0.18, 0.18);
\draw[edge] (D3) -- ++( 0.18, 0.18);

\draw[edge] (D4) -- ++(-0.18,-0.18);
\draw[edge] (D4) -- ++( 0.18,-0.18);

\draw[edge] (D5) -- ++(-0.18,0.18);

\draw[edge] (D6) -- ++(-0.18,-0.18);
\draw[edge] (D6) -- ++( 0.18,-0.18);

\draw[edge] (D7) -- ++(-0.18, 0.18);
\draw[edge] (D7) -- ++( 0.18, 0.18);

\draw[edge] (D8) -- ++(-0.18,-0.18);
\draw[edge] (D8) -- ++( 0.18,-0.18);

\draw[edge] (D9) -- ++(-0.18, 0.18);
\draw[edge] (D9) -- ++( 0.18,-0.18);

\foreach \i in {1,5,9}{
    \draw[connection] (D\i) -- (U\i);
}

\draw[greenline,thick,dashed]
    (-0.35,0.95) -- (8.0,0.95);

\foreach \i in {1,...,9}{
    \node[vertex] at (U\i) {};
}

\foreach \i in {1,...,9}{
    \node[vertex] at (D\i) {};
}

\foreach \i in {1,...,9}{
    \node[above=4pt] at (U\i) {$u_{\i}$};
}

\foreach \i in {1,...,9}{
    \node[below=4pt] at (D\i) {$d_{\i}$};
}

\node[anchor=west] at (7.82,1.42) {$M_U$};
\node[anchor=west] at (7.82,0.48) {$M_D$};

\end{tikzpicture}
\caption{Initial type II molecule}
\end{figure}
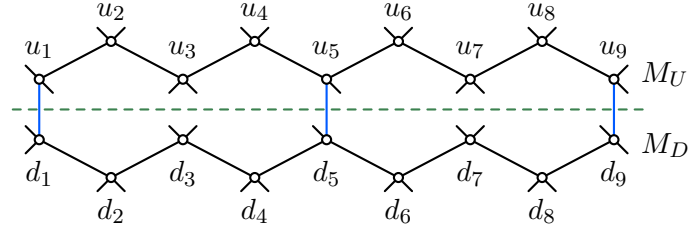
In the following figures, we draw the evolutive molecule on the left and its associated Kruskal spanning tree on the right.
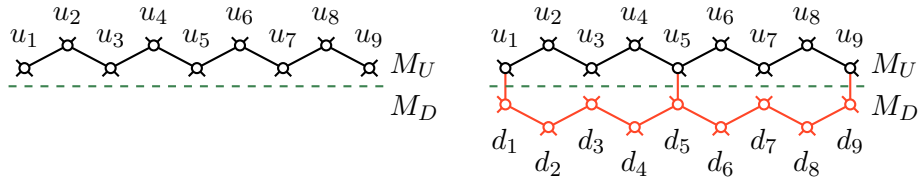
\begin{figure}[H]
\centering
\begin{tikzpicture}[scale=0.6]

\definecolor{greenline}{RGB}{60,130,80}
\definecolor{blueconn}{RGB}{0,90,255}
\definecolor{rednote}{RGB}{255,70,40}

\def\dx{0.95}

\begin{scope}[xshift=0cm]

\foreach \i in {1,...,9}{
    \pgfmathsetmacro{\x}{(\i-1)*\dx}
    \ifodd\i
        \coordinate (Lu\i) at (\x,1.35);
    \else
        \coordinate (Lu\i) at (\x,1.85);
    \fi
}

\draw[black,thick]
    (Lu1)--(Lu2)--(Lu3)--(Lu4)--(Lu5)--(Lu6)--(Lu7)--(Lu8)--(Lu9);

\draw[black,thick] (Lu1) -- ++(-0.18, 0.18);
\draw[black,thick] (Lu1) -- ++( 0.18,-0.18);
\draw[black,thick] (Lu1) -- ++(-0.18,-0.18);

\draw[black,thick] (Lu2) -- ++(-0.18, 0.18);
\draw[black,thick] (Lu2) -- ++( 0.18, 0.18);

\draw[black,thick] (Lu3) -- ++(-0.18,-0.18);
\draw[black,thick] (Lu3) -- ++( 0.18,-0.18);

\draw[black,thick] (Lu4) -- ++(-0.18, 0.18);
\draw[black,thick] (Lu4) -- ++( 0.18, 0.18);

\draw[black,thick] (Lu5) -- ++( 0.18,-0.18);
\draw[black,thick] (Lu5) -- ++(-0.18,-0.18);

\draw[black,thick] (Lu6) -- ++(-0.18, 0.18);
\draw[black,thick] (Lu6) -- ++( 0.18, 0.18);

\draw[black,thick] (Lu7) -- ++(-0.18,-0.18);
\draw[black,thick] (Lu7) -- ++( 0.18,-0.18);

\draw[black,thick] (Lu8) -- ++(-0.18, 0.18);
\draw[black,thick] (Lu8) -- ++( 0.18, 0.18);

\draw[black,thick] (Lu9) -- ++( 0.18, 0.18);
\draw[black,thick] (Lu9) -- ++( 0.18,-0.18);
\draw[black,thick] (Lu9) -- ++(-0.18,-0.18);

\draw[greenline,thick,dashed]
    (-0.35,0.95) -- (8.0,0.95);

\foreach \i in {1,...,9}{
    \filldraw[white,draw=black,thick] (Lu\i) circle (0.11);
}

\foreach \i in {1,...,9}{
    \node[above=4pt] at (Lu\i) {$u_{\i}$};
}

\node[anchor=west] at (7.82,1.42) {$M_U$};
\node[anchor=west] at (7.82,0.48) {$M_D$};

\end{scope}\begin{scope}[xshift=10.6cm]

\foreach \i in {1,...,9}{
    \pgfmathsetmacro{\x}{(\i-1)*\dx}
    \ifodd\i
        \coordinate (Ru\i) at (\x,1.35);
    \else
        \coordinate (Ru\i) at (\x,1.85);
    \fi
}

\foreach \i in {1,...,9}{
    \pgfmathsetmacro{\x}{(\i-1)*\dx}
    \ifodd\i
        \coordinate (Rd\i) at (\x,0.55);
    \else
        \coordinate (Rd\i) at (\x,0.05);
    \fi
}

\draw[black,thick]
    (Ru1)--(Ru2)--(Ru3)--(Ru4)--(Ru5)--(Ru6)--(Ru7)--(Ru8)--(Ru9);

\draw[rednote,thick]
    (Rd1)--(Rd2)--(Rd3)--(Rd4)--(Rd5)--(Rd6)--(Rd7)--(Rd8)--(Rd9);

\draw[black,thick] (Ru1) -- ++(-0.18, 0.18);
\draw[black,thick] (Ru1) -- ++( 0.18,-0.18);

\draw[black,thick] (Ru2) -- ++(-0.18, 0.18);
\draw[black,thick] (Ru2) -- ++( 0.18, 0.18);

\draw[black,thick] (Ru3) -- ++(-0.18,-0.18);
\draw[black,thick] (Ru3) -- ++( 0.18,-0.18);

\draw[black,thick] (Ru4) -- ++(-0.18, 0.18);
\draw[black,thick] (Ru4) -- ++( 0.18, 0.18);

\draw[black,thick] (Ru5) -- ++(0.18,-0.18);

\draw[black,thick] (Ru6) -- ++(-0.18, 0.18);
\draw[black,thick] (Ru6) -- ++( 0.18, 0.18);

\draw[black,thick] (Ru7) -- ++(-0.18,-0.18);
\draw[black,thick] (Ru7) -- ++( 0.18,-0.18);

\draw[black,thick] (Ru8) -- ++(-0.18, 0.18);
\draw[black,thick] (Ru8) -- ++( 0.18, 0.18);

\draw[black,thick] (Ru9) -- ++(0.18, 0.18);
\draw[black,thick] (Ru9) -- ++(0.18,-0.18);

\draw[rednote,thick] (Rd1) -- ++(-0.18, 0.18);
\draw[rednote,thick] (Rd1) -- ++(-0.18,-0.18);

\draw[rednote,thick] (Rd2) -- ++(-0.18,-0.18);
\draw[rednote,thick] (Rd2) -- ++( 0.18,-0.18);

\draw[rednote,thick] (Rd3) -- ++(-0.18, 0.18);
\draw[rednote,thick] (Rd3) -- ++( 0.18, 0.18);

\draw[rednote,thick] (Rd4) -- ++(-0.18,-0.18);
\draw[rednote,thick] (Rd4) -- ++( 0.18,-0.18);

\draw[rednote,thick] (Rd5) -- ++(-0.18,0.18);

\draw[rednote,thick] (Rd6) -- ++(-0.18,-0.18);
\draw[rednote,thick] (Rd6) -- ++( 0.18,-0.18);

\draw[rednote,thick] (Rd7) -- ++(-0.18, 0.18);
\draw[rednote,thick] (Rd7) -- ++( 0.18, 0.18);

\draw[rednote,thick] (Rd8) -- ++(-0.18,-0.18);
\draw[rednote,thick] (Rd8) -- ++( 0.18,-0.18);

\draw[rednote,thick] (Rd9) -- ++(-0.18, 0.18);
\draw[rednote,thick] (Rd9) -- ++( 0.18,-0.18);

\foreach \i in {1,5,9}{
    \draw[rednote,thick] (Rd\i) -- (Ru\i);
}

\draw[greenline,thick,dashed]
    (-0.35,0.95) -- (8.0,0.95);

\foreach \i in {1,...,9}{
    \filldraw[white,draw=black,thick] (Ru\i) circle (0.11);
}

\foreach \i in {1,...,9}{
    \filldraw[white,draw=rednote,thick] (Rd\i) circle (0.11);
}

\foreach \i in {1,...,9}{
    \node[above=4pt] at (Ru\i) {$u_{\i}$};
    \node[below=5pt] at (Rd\i) {$d_{\i}$};
}

\node[anchor=west] at (7.82,1.42) {$M_U$};
\node[anchor=west] at (7.82,0.48) {$M_D$};

\end{scope}

\end{tikzpicture}
\caption{Molecule after cutting $M_D$}
\end{figure}
\begin{figure}[H]
\begin{tikzpicture}[scale=0.6]

\definecolor{greenline}{RGB}{60,130,80}
\definecolor{blueconn}{RGB}{0,90,255}
\definecolor{rednote}{RGB}{255,70,40}

\def\dx{0.95}

\begin{scope}[xshift=0cm]

\foreach \i in {1,...,9}{
    \pgfmathsetmacro{\x}{(\i-1)*\dx}
    \ifodd\i
        \coordinate (Lu\i) at (\x,1.35);
    \else
        \coordinate (Lu\i) at (\x,1.85);
    \fi
}

\draw[black,thick]
    (Lu6)--(Lu7)--(Lu8)--(Lu9);

\draw[black,thick] (Lu6) -- ++(-0.18, 0.18);
\draw[black,thick] (Lu6) -- ++( 0.18, 0.18);
\draw[black,thick] (Lu6) -- ++(-0.18,-0.18);

\draw[black,thick] (Lu7) -- ++(-0.18,-0.18);
\draw[black,thick] (Lu7) -- ++( 0.18,-0.18);

\draw[black,thick] (Lu8) -- ++(-0.18, 0.18);
\draw[black,thick] (Lu8) -- ++( 0.18, 0.18);

\draw[black,thick] (Lu9) -- ++( 0.18, 0.18);
\draw[black,thick] (Lu9) -- ++( 0.18,-0.18);
\draw[black,thick] (Lu9) -- ++(-0.18,-0.18);

\draw[greenline,thick,dashed]
    (-0.35,0.95) -- (8.0,0.95);

\foreach \i in {6,...,9}{
    \filldraw[white,draw=black,thick] (Lu\i) circle (0.11);
}

\foreach \i in {6,...,9}{
    \node[above=4pt] at (Lu\i) {$u_{\i}$};
}

\node[anchor=west] at (7.82,1.42) {$M_U$};
\node[anchor=west] at (7.82,0.48) {$M_D$};

\end{scope}

\begin{scope}[xshift=10.6cm]

\foreach \i in {1,...,9}{
    \pgfmathsetmacro{\x}{(\i-1)*\dx}
    \ifodd\i
        \coordinate (Ru\i) at (\x,1.35);
    \else
        \coordinate (Ru\i) at (\x,1.85);
    \fi
}

\foreach \i in {1,...,9}{
    \pgfmathsetmacro{\x}{(\i-1)*\dx}
    \ifodd\i
        \coordinate (Rd\i) at (\x,0.55);
    \else
        \coordinate (Rd\i) at (\x,0.05);
    \fi
}

\draw[rednote,thick]
    (Ru1)--(Ru2)--(Ru3)--(Ru4);

\draw[black,thick]
    (Ru4)--(Ru5);

\draw[rednote,thick]
    (Ru5)--(Ru6);

\draw[black,thick]
    (Ru6)--(Ru7)--(Ru8)--(Ru9);

\draw[rednote,thick] (Ru1) -- ++(-0.18, 0.18);
\draw[rednote,thick] (Ru1) -- ++( 0.18,-0.18);

\draw[rednote,thick] (Ru2) -- ++(-0.18, 0.18);
\draw[rednote,thick] (Ru2) -- ++( 0.18, 0.18);

\draw[rednote,thick] (Ru3) -- ++(-0.18,-0.18);
\draw[rednote,thick] (Ru3) -- ++( 0.18,-0.18);

\draw[rednote,thick] (Ru4) -- ++(-0.18, 0.18);
\draw[rednote,thick] (Ru4) -- ++( 0.18, 0.18);

\draw[rednote,thick] (Ru5) -- ++(0.18,-0.18);

\draw[black,thick] (Ru6) -- ++(-0.18, 0.18);
\draw[black,thick] (Ru6) -- ++( 0.18, 0.18);

\draw[black,thick] (Ru7) -- ++(-0.18,-0.18);
\draw[black,thick] (Ru7) -- ++( 0.18,-0.18);

\draw[black,thick] (Ru8) -- ++(-0.18, 0.18);
\draw[black,thick] (Ru8) -- ++( 0.18, 0.18);

\draw[black,thick] (Ru9) -- ++(0.18, 0.18);
\draw[black,thick] (Ru9) -- ++(0.18,-0.18);

\draw[rednote,thick]
    (Rd1)--(Rd2)--(Rd3)--(Rd4)--(Rd5)--(Rd6)--(Rd7)--(Rd8)--(Rd9);

\draw[rednote,thick] (Rd1) -- ++(-0.18, 0.18);
\draw[rednote,thick] (Rd1) -- ++(-0.18,-0.18);

\draw[rednote,thick] (Rd2) -- ++(-0.18,-0.18);
\draw[rednote,thick] (Rd2) -- ++( 0.18,-0.18);

\draw[rednote,thick] (Rd3) -- ++(-0.18, 0.18);
\draw[rednote,thick] (Rd3) -- ++( 0.18, 0.18);

\draw[rednote,thick] (Rd4) -- ++(-0.18,-0.18);
\draw[rednote,thick] (Rd4) -- ++( 0.18,-0.18);

\draw[rednote,thick] (Rd5) -- ++(-0.18,0.18);

\draw[rednote,thick] (Rd6) -- ++(-0.18,-0.18);
\draw[rednote,thick] (Rd6) -- ++( 0.18,-0.18);

\draw[rednote,thick] (Rd7) -- ++(-0.18, 0.18);
\draw[rednote,thick] (Rd7) -- ++( 0.18, 0.18);

\draw[rednote,thick] (Rd8) -- ++(-0.18,-0.18);
\draw[rednote,thick] (Rd8) -- ++( 0.18,-0.18);

\draw[rednote,thick] (Rd9) -- ++(-0.18, 0.18);
\draw[rednote,thick] (Rd9) -- ++( 0.18,-0.18);

\draw[rednote,thick] (Rd1) -- (Ru1);
\draw[rednote,thick] (Rd5) -- (Ru5);
\draw[rednote,thick] (Rd9) -- (Ru9);

\draw[greenline,thick,dashed]
    (-0.35,0.95) -- (8.0,0.95);

\foreach \i in {1,...,5}{
    \filldraw[white,draw=rednote,thick] (Ru\i) circle (0.11);
}

\foreach \i in {6,...,9}{
    \filldraw[white,draw=black,thick] (Ru\i) circle (0.11);
}

\foreach \i in {1,...,9}{
    \filldraw[white,draw=rednote,thick] (Rd\i) circle (0.11);
}

\foreach \i in {1,...,9}{
    \node[above=4pt] at (Ru\i) {$u_{\i}$};
    \node[below=5pt] at (Rd\i) {$d_{\i}$};
}

\node[anchor=west] at (7.82,1.42) {$M_U$};
\node[anchor=west] at (7.82,0.48) {$M_D$};

\end{scope}

\end{tikzpicture}
\caption{Molecule after applying steps 2,3,4 to $A=\{u_1,u_2,u_3,u_4,u_5\}$}
\end{figure}
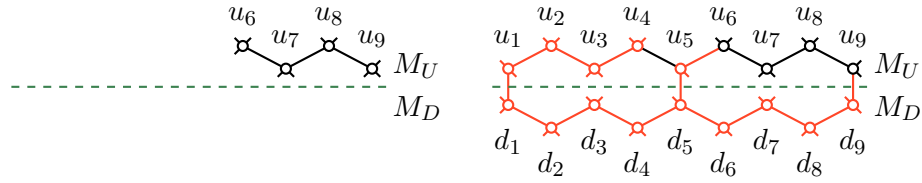
\begin{figure}[H]
\centering
\begin{tikzpicture}[scale=0.6]

\definecolor{greenline}{RGB}{60,130,80}
\definecolor{blueconn}{RGB}{0,90,255}
\definecolor{rednote}{RGB}{255,70,40}

\def\dx{0.95}

\begin{scope}[xshift=0cm]

\draw[greenline,thick,dashed]
    (-0.35,0.95) -- (8.0,0.95);

\node[anchor=west] at (7.82,1.42) {$M_U$};
\node[anchor=west] at (7.82,0.48) {$M_D$};

\end{scope}

\begin{scope}[xshift=10.6cm]

\foreach \i in {1,...,9}{
    \pgfmathsetmacro{\x}{(\i-1)*\dx}
    \ifodd\i
        \coordinate (Ru\i) at (\x,1.35);
    \else
        \coordinate (Ru\i) at (\x,1.85);
    \fi
}

\foreach \i in {1,...,9}{
    \pgfmathsetmacro{\x}{(\i-1)*\dx}
    \ifodd\i
        \coordinate (Rd\i) at (\x,0.55);
    \else
        \coordinate (Rd\i) at (\x,0.05);
    \fi
}

\draw[rednote,thick]
    (Ru1)--(Ru2)--(Ru3)--(Ru4);

\draw[black,thick]
    (Ru4)--(Ru5);

\draw[rednote,thick]
    (Ru5)--(Ru6)--(Ru7)--(Ru8);

\draw[black,thick]
    (Ru8)--(Ru9);

\draw[rednote,thick] (Ru1) -- ++(-0.18, 0.18);
\draw[rednote,thick] (Ru1) -- ++( 0.18,-0.18);

\draw[rednote,thick] (Ru2) -- ++(-0.18, 0.18);
\draw[rednote,thick] (Ru2) -- ++( 0.18, 0.18);

\draw[rednote,thick] (Ru3) -- ++(-0.18,-0.18);
\draw[rednote,thick] (Ru3) -- ++( 0.18,-0.18);

\draw[rednote,thick] (Ru4) -- ++(-0.18, 0.18);
\draw[rednote,thick] (Ru4) -- ++( 0.18, 0.18);

\draw[rednote,thick] (Ru5) -- ++(0.18,-0.18);

\draw[rednote,thick] (Ru6) -- ++(-0.18, 0.18);
\draw[rednote,thick] (Ru6) -- ++( 0.18, 0.18);

\draw[rednote,thick] (Ru7) -- ++(-0.18,-0.18);
\draw[rednote,thick] (Ru7) -- ++( 0.18,-0.18);

\draw[rednote,thick] (Ru8) -- ++(-0.18, 0.18);
\draw[rednote,thick] (Ru8) -- ++( 0.18, 0.18);

\draw[rednote,thick] (Ru9) -- ++(0.18, 0.18);
\draw[rednote,thick] (Ru9) -- ++(0.18,-0.18);

\draw[rednote,thick]
    (Rd1)--(Rd2)--(Rd3)--(Rd4)--(Rd5)--(Rd6)--(Rd7)--(Rd8)--(Rd9);

\draw[rednote,thick] (Rd1) -- ++(-0.18, 0.18);
\draw[rednote,thick] (Rd1) -- ++(-0.18,-0.18);

\draw[rednote,thick] (Rd2) -- ++(-0.18,-0.18);
\draw[rednote,thick] (Rd2) -- ++( 0.18,-0.18);

\draw[rednote,thick] (Rd3) -- ++(-0.18, 0.18);
\draw[rednote,thick] (Rd3) -- ++( 0.18, 0.18);

\draw[rednote,thick] (Rd4) -- ++(-0.18,-0.18);
\draw[rednote,thick] (Rd4) -- ++( 0.18,-0.18);

\draw[rednote,thick] (Rd5) -- ++(-0.18,0.18);

\draw[rednote,thick] (Rd6) -- ++(-0.18,-0.18);
\draw[rednote,thick] (Rd6) -- ++( 0.18,-0.18);

\draw[rednote,thick] (Rd7) -- ++(-0.18, 0.18);
\draw[rednote,thick] (Rd7) -- ++( 0.18, 0.18);

\draw[rednote,thick] (Rd8) -- ++(-0.18,-0.18);
\draw[rednote,thick] (Rd8) -- ++( 0.18,-0.18);

\draw[rednote,thick] (Rd9) -- ++(-0.18, 0.18);
\draw[rednote,thick] (Rd9) -- ++( 0.18,-0.18);

\draw[rednote,thick] (Rd1) -- (Ru1);
\draw[rednote,thick] (Rd5) -- (Ru5);
\draw[rednote,thick] (Rd9) -- (Ru9);

\draw[greenline,thick,dashed]
    (-0.35,0.95) -- (8.0,0.95);

\foreach \i in {1,...,9}{
    \filldraw[white,draw=rednote,thick] (Ru\i) circle (0.11);
}

\foreach \i in {1,...,9}{
    \filldraw[white,draw=rednote,thick] (Rd\i) circle (0.11);
}

\foreach \i in {1,...,9}{
    \node[above=4pt] at (Ru\i) {$u_{\i}$};
    \node[below=5pt] at (Rd\i) {$d_{\i}$};
}

\node[anchor=west] at (7.82,1.42) {$M_U$};
\node[anchor=west] at (7.82,0.48) {$M_D$};

\end{scope}

\end{tikzpicture}
\caption{Molecule after applying steps 2,3,4 to $A=\{u_6,u_7,u_8,u_9\}$}
\end{figure}
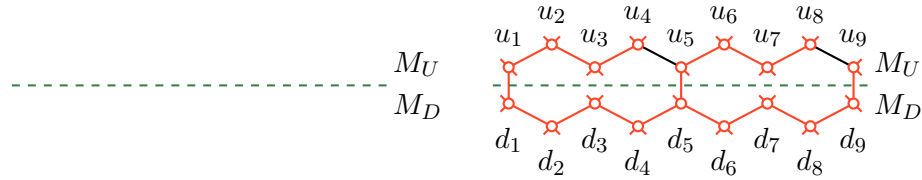

\subsection{Toy model III}
\begin{definition}[Toy model III]
A two-layer molecule $M=M_U\cup M_D$ is said to be a
\emph{toy model III} if $\#2-\mathrm{conn}\gg\lvert X\rvert$.
\end{definition}

In this case the molecule contains a large number of $2$-connections,
which already provide a favorable combinatorial structure.

\begin{proof}[of Proposition~\ref{main_proposition} for Toy model III]
We perform the proof of the previous proposition in Section 3.1 by reviewing
all the steps of the algorithm given in \cite[Definition 11.21]{DHM25a}.
Let $M$ be a toy model III. We have the following cutting sequence.
\begin{enumerate}
\item If $M_D$ has a degree-$2$ atom $\mathfrak{n}$, then cut it as free, add it to the acyclic graph $G$ with the maximum amount of bonds linked to $\mathfrak{n}$. Repeat until $M_D$ does not have degree-$2$ atoms.
\begin{figure}[H]
\centering
\begin{tikzpicture}[scale=0.7]

\draw[greenline,thick,dashed] (-1.6,1.5) -- (1.6,1.5);

\node[cutatom] (n) at (0,0.65) {};

\draw[red,thick] (n) -- (-0.55,1.2);
\draw[red,thick] (n) -- (0.55,0.1);

\node[right=4pt] at (n) {$\mathfrak{n}$};

\node[right] at (1.35,2.05) {$M_U$};
\node[right] at (1.35,0.65) {$M_D$};

\end{tikzpicture}
\caption{Step 1}
\end{figure}
\item Choose a highest degree-$3$ atom $\mathfrak{n}$ in $M_D$ that has not been
cut (or a highest atom if $M_D$ has no degree-$3$ atoms). Let $Z_{\mathfrak{n}}$ be the set of ancestors of $\mathfrak{n}$ in $M_D$. This $\mathfrak{n}$ and $Z_\mathfrak{n}$ are fixed until all atoms in $Z_\mathfrak{n}$ have been cut.

\item Starting from $\mathfrak{n}$, each time choose a lowest atom $\mathfrak{m}$ in $Z_\mathfrak{n}$ that has not been cut. If $\mathfrak{m}$ has degree $3$ and has a parent $\mathfrak{p}\in M_U$ that also has degree $3$, then cut $\{\mathfrak{m},\mathfrak{p}\}$ as free; otherwise cut $\mathfrak{m}$ as free. In any cases, add the cut atoms to the acyclic graph $G$ with the maximum amount of bonds linked to them such that there is no cycle created in $G$. Repeat until all atoms in $Z_\mathfrak{n}$ have been cut, then go to (1)--(2) and choose the next $\mathfrak{n}$, and so on.
\begin{figure}[H]
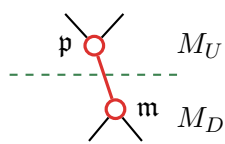

\centering

\caption{Step 3, first case}
\end{figure}

\item When all atoms in $M_D$ have been cut, finish off $M_U$ by
applying the UP algorithm.
\end{enumerate}
\end{proof}
All the steps may cut good, normal or bad submolecules, as described in this table.
\begin{table}[H]
\centering
\renewcommand{\arraystretch}{1.3}

\resizebox{\textwidth}{!}{
\begin{tabular}{|c|c|c|c|}
\hline
\textbf{Step} & \textbf{Cut} & \textbf{Location} & \textbf{Type} \\
\hline

(1)
& $\{2\}$
& $M_D$
& Normal \\
\hline

(2)
& Selection of $\mathfrak{n}$ and $Z_{\mathfrak{n}}$
& $M_D$
& -- \\
\hline

(3)
& $\{33\}$ with one atom in $M_D$ and one atom in $M_U$
& Across $M_D$ and $M_U$
& Good \\
& $\{3\}$
& $M_D$
& Normal \\
& $\{4\}$
& $M_D$
& Bad \\
\hline

(4)
& $\{33\}$ produced by UP in $M_U$
& $M_U$
& Good \\
& $\{2\}$ or $\{3\}$ produced by UP in $M_U$
& $M_U$
& Normal \\
\hline

\end{tabular}
}

\caption{Classification and location of the cuts produced by the toy model III algorithm.}
\end{table}

We can note that the only optimisation step is the step 3.
\begin{table}[H]
\centering
\renewcommand{\arraystretch}{1.3}

\resizebox{\textwidth}{!}{
\begin{tabular}{|c|c|c|c|}
\hline
\textbf{Priority} & \textbf{Cut} & \textbf{Location} & \textbf{Type} \\
\hline
(1)
& $\{33\}$ with one atom in $M_D$ and one atom in $M_U$
& Across $M_D$ and $M_U$
& Good \\
& $\{3\}$
& $M_D$
& Normal \\
& $\{4\}$
& $M_D$
& Bad \\
\hline
\end{tabular}
}
\caption{Classification and location of the cuts produced by the toy model III algorithm.}
\end{table}

Here is an example of type III molecule reduction, with the construction of the associated spanning tree.
\begin{figure}[H]
\centering
%
\caption{Initial type III molecule}
\end{figure}
In the following figures, we draw the evolutive molecule on the left and its associated Kruskal spanning tree on the right.
\begin{figure}[H]
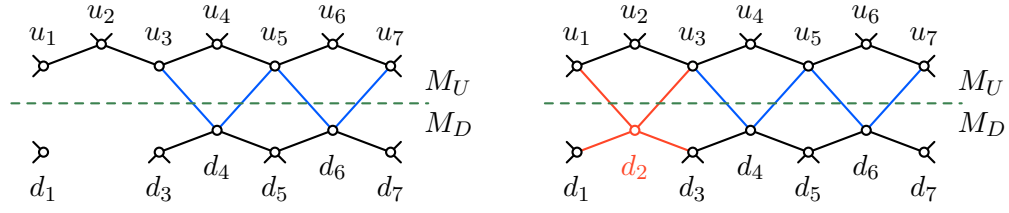

\centering
%
\caption{Molecule after applying step 3 on the $\{4\}$-atom $d_2$ with $\mathfrak{n}=d_2$ (here, $Z_\mathfrak{n}=\{d_2\}$)}
\end{figure}
\begin{figure}[H]
\centering
%
\caption{Molecule after applying in a new loop step 3 on the $\{3\}$-atom $d_3$ with $\mathfrak{n}=\{d_3\}$ (here $Z_\mathfrak{n}=\{d_3,d_4\}$)}
\end{figure}
\begin{figure}[H]
\centering
%
\caption{Molecule after applying step 3 on the $\{33\}$-submolecule $u_3-d_4$}
\end{figure}
\begin{figure}[H]
\centering
%
\caption{Molecule after applying in a new loop step 3 on the $\{3\}$-atom $d_5$ with $\mathfrak{n}=d_5$ (here, $Z_\mathfrak{n}=\{d_5,d_6\}$)}
\end{figure}
\begin{figure}[H]
\centering
%
\caption{Molecule after applying step 3 on the $\{33\}$-submolecule $u_5-d_6$}
\end{figure}


\begin{thebibliography}{DHM25b}

	\expandafter\ifx\csname url\endcsname\relax
	\def\url#1{\texttt{#1}}\fi
	\expandafter\ifx\csname urlprefix\endcsname\relax
	\def\urlprefix{URL }\fi
	\expandafter\ifx\csname href\endcsname\relax
	\def\href#1#2{#2}\fi
	\expandafter\ifx\csname burlalt\endcsname\relax
	\def\burlalt#1#2{\href{#2}{\texttt{#1}}}\fi


	\bibitem{A75}
	{\rm R.~K. Alexander}.
	\newblock \emph{The Infinite Hard-Sphere System}.
	\newblock Ph.D. thesis, University of California, Berkeley, (1975).
	\newblock
	\burlalt{eScholarship}{https://escholarship.org/uc/item/77b0c501}.


	\bibitem{BC26}
	{\rm Y.~Bruned, V.~Clarisse}.
	\newblock \emph{Kruskal-style algorithm for cubic Schr\"odinger equation molecule reduction}.
	\newblock Preprint, (2026).
	\newblock
	\burlalt{arXiv:2603.23298}{https://arxiv.org/abs/2603.23298}.


	\bibitem{BGSR16}
	{\rm T.~Bodineau, I.~Gallagher, L.~Saint-Raymond}.
	\newblock \emph{The Brownian motion as the limit of a deterministic system of hard-spheres}.
	\newblock Invent. Math. \textbf{203}, no.~2, (2016), 493--553.
	\newblock
	\burlalt{doi:10.1007/s00222-015-0593-9}{https://dx.doi.org/10.1007/s00222-015-0593-9}.


	\bibitem{BGSS22}
	{\rm T.~Bodineau, I.~Gallagher, L.~Saint-Raymond, S.~Simonella}.
	\newblock \emph{Cluster expansion for a dilute hard sphere gas dynamics}.
	\newblock J. Math. Phys. \textbf{63}, no.~7, (2022), Paper No.~073301.
	\newblock
	\burlalt{doi:10.1063/5.0091199}{https://dx.doi.org/10.1063/5.0091199}.


	\bibitem{BGSS23}
	{\rm T.~Bodineau, I.~Gallagher, L.~Saint-Raymond, S.~Simonella}.
	\newblock \emph{Statistical dynamics of a hard sphere gas: fluctuating Boltzmann equation and large deviations}.
	\newblock Ann. of Math. (2) \textbf{198}, no.~3, (2023), 1047--1201.
	\newblock
	\burlalt{doi:10.4007/annals.2023.198.3.3}{https://dx.doi.org/10.4007/annals.2023.198.3.3}.


	\bibitem{BGSS26}
	{\rm T.~Bodineau, I.~Gallagher, L.~Saint-Raymond, S.~Simonella}.
	\newblock \emph{Derivation of the Boltzmann equation from hard-sphere dynamics
	(after Y. Deng, Z. Hani, and X. Ma)}.
	\newblock S\'eminaire Bourbaki, 78e ann\'ee, 2025--2026,
	Exp.~no.~1247, (2026).
	\newblock
	\burlalt{arXiv:2602.04407}{https://arxiv.org/abs/2602.04407}.


	\bibitem{C72}
	{\rm C.~Cercignani}.
	\newblock \emph{On the Boltzmann equation for rigid spheres}.
	\newblock Transport Theory Statist. Phys. \textbf{2}, no.~3, (1972), 211--225.
	\newblock
	\burlalt{doi:10.1080/00411457208232538}{https://dx.doi.org/10.1080/00411457208232538}.


	\bibitem{CIP94}
	{\rm C.~Cercignani, R.~Illner, M.~Pulvirenti}.
	\newblock \emph{The Mathematical Theory of Dilute Gases}.
	\newblock Applied Mathematical Sciences, \textbf{106},
	Springer, New York, (1994).
	\newblock
	\burlalt{doi:10.1007/978-1-4419-8524-8}{https://dx.doi.org/10.1007/978-1-4419-8524-8}.


	\bibitem{D18}
	{\rm R.~Denlinger}.
	\newblock \emph{The propagation of chaos for a rarefied gas of hard spheres in the whole space}.
	\newblock Arch. Ration. Mech. Anal. \textbf{229}, no.~2, (2018), 885--952.
	\newblock
	\burlalt{doi:10.1007/s00205-018-1229-1}{https://dx.doi.org/10.1007/s00205-018-1229-1}.


	\bibitem{DEx1}
	{\rm Y.~Deng}.
	\newblock \emph{Overview of the proof}.
	\newblock Expository notes on the long-time derivation of the Boltzmann equation.
	\newblock
	\burlalt{Yu Deng's expository notes}{https://drive.google.com/file/d/1jxj88jb6ckE04rw4-snC9pvy39aON5OG/view?usp=share_link}.


	\bibitem{DEx2}
	{\rm Y.~Deng}.
	\newblock \emph{Toy model algorithms}.
	\newblock Expository notes on the long-time derivation of the Boltzmann equation.
	\newblock
	\burlalt{Yu Deng's expository notes}{https://drive.google.com/file/d/1_CRocvk5zLULejg-GRsvdi2tgAB7JTb9/view?usp=share_link}.


	\bibitem{DH23a}
	{\rm Y.~Deng, Z.~Hani}.
	\newblock \emph{Full derivation of the wave kinetic equation}.
	\newblock Invent. Math. \textbf{233}, no.~2, (2023), 543--724.
	\newblock
	\burlalt{doi:10.1007/s00222-023-01189-2}{https://dx.doi.org/10.1007/s00222-023-01189-2}.


	\bibitem{DH23b}
	{\rm Y.~Deng, Z.~Hani}.
	\newblock \emph{Long time justification of wave turbulence theory}.
	\newblock Preprint, (2023).
	\newblock
	\burlalt{arXiv:2311.10082}{https://arxiv.org/abs/2311.10082}.


	\bibitem{DH26}
	{\rm Y.~Deng, Z.~Hani}.
	\newblock \emph{Propagation of chaos and higher order statistics in wave kinetic theory}.
	\newblock J. Eur. Math. Soc. (JEMS) \textbf{28}, no.~2, (2026), 673--733.
	\newblock
	\burlalt{doi:10.4171/JEMS/1488}{https://dx.doi.org/10.4171/JEMS/1488}.


	\bibitem{DHM25a}
	{\rm Y.~Deng, Z.~Hani, X.~Ma}.
	\newblock \emph{Long time derivation of the Boltzmann equation from hard sphere dynamics}.
	\newblock To appear in Ann. of Math.
	\newblock
	\burlalt{arXiv:2408.07818}{https://arxiv.org/abs/2408.07818}.


	\bibitem{DHM25b}
	{\rm Y.~Deng, Z.~Hani, X.~Ma}.
	\newblock \emph{Hilbert's sixth problem: derivation of fluid equations via Boltzmann's kinetic theory}.
	\newblock Preprint, (2025).
	\newblock
	\burlalt{arXiv:2503.01800}{https://arxiv.org/abs/2503.01800}.


	\bibitem{G58}
	{\rm H.~Grad}.
	\newblock \emph{Principles of the kinetic theory of gases}.
	\newblock In \emph{Handbuch der Physik}, Vol.~12,
	Springer, (1958), 205--294.
	\newblock
	\burlalt{doi:10.1007/978-3-642-45892-7_3}{https://dx.doi.org/10.1007/978-3-642-45892-7_3}.


	\bibitem{GST14}
	{\rm I.~Gallagher, L.~Saint-Raymond, B.~Texier}.
	\newblock \emph{From Newton to Boltzmann: Hard Spheres and Short-range Potentials}.
	\newblock Z\"urich Lectures in Advanced Mathematics,
	European Mathematical Society, (2014).
	\newblock
	\burlalt{doi:10.4171/129}{https://dx.doi.org/10.4171/129}.


	\bibitem{IP89}
	{\rm R.~Illner, M.~Pulvirenti}.
	\newblock \emph{Global validity of the Boltzmann equation for two- and three-dimensional rare gas in vacuum: Erratum and improved result}.
	\newblock Comm. Math. Phys. \textbf{121}, no.~1, (1989), 143--146.
	\newblock
	\burlalt{doi:10.1007/BF01218628}{https://dx.doi.org/10.1007/BF01218628}.


	\bibitem{K56}
	{\rm J.~B. Kruskal}.
	\newblock \emph{On the shortest spanning subtree of a graph and the traveling salesman problem}.
	\newblock Proc. Amer. Math. Soc. \textbf{7}, (1956), 48--50.
	\newblock
	\burlalt{doi:10.1090/S0002-9939-1956-0078686-7}{https://dx.doi.org/10.1090/S0002-9939-1956-0078686-7}.


	\bibitem{L75}
	{\rm O.~E. Lanford III}.
	\newblock \emph{Time evolution of large classical systems}.
	\newblock In \emph{Dynamical Systems, Theory and Applications},
	Lecture Notes in Physics, \textbf{38}, Springer, (1975), 1--111.
	\newblock
	\burlalt{doi:10.1007/3-540-07171-7_1}{https://dx.doi.org/10.1007/3-540-07171-7_1}.


	\bibitem{PS17}
	{\rm M.~Pulvirenti, S.~Simonella}.
	\newblock \emph{The Boltzmann--Grad limit of a hard sphere system: analysis of the correlation error}.
	\newblock Invent. Math. \textbf{207}, no.~3, (2017), 1135--1237.
	\newblock
	\burlalt{doi:10.1007/s00222-016-0682-4}{https://dx.doi.org/10.1007/s00222-016-0682-4}.


	\bibitem{PSS14}
	{\rm M.~Pulvirenti, C.~Saffirio, S.~Simonella}.
	\newblock \emph{On the validity of the Boltzmann equation for short range potentials}.
	\newblock Rev. Math. Phys. \textbf{26}, no.~2, (2014), 1450001.
	\newblock
	\burlalt{doi:10.1142/S0129055X14500019}{https://dx.doi.org/10.1142/S0129055X14500019}.


	\bibitem{R91}
	{\rm V.~Rivasseau}.
	\newblock \emph{From Perturbative to Constructive Renormalization}.
	\newblock Princeton University Press, (1991).
	\newblock
	\burlalt{doi:10.1515/9781400862085}{https://dx.doi.org/10.1515/9781400862085}.


	\bibitem{RW14}
	{\rm V.~Rivasseau, Z.~Wang}.
	\newblock \emph{How to Resum Feynman Graphs}.
	\newblock Ann. Henri Poincar\'e \textbf{15}, (2014), 2069--2083.
	\newblock
	\burlalt{doi:10.1007/s00023-013-0299-8}{https://dx.doi.org/10.1007/s00023-013-0299-8}.


	\bibitem{S91}
	{\rm H.~Spohn}.
	\newblock \emph{Large Scale Dynamics of Interacting Particles}.
	\newblock Texts and Monographs in Physics,
	Springer, Berlin, (1991).
	\newblock
	\burlalt{doi:10.1007/978-3-642-84371-6}{https://dx.doi.org/10.1007/978-3-642-84371-6}.


	\bibitem{V25}
	{\rm K.~D. Vassilev}.
	\newblock \emph{One-Dimensional Wave Kinetic Theory}.
	\newblock Commun. Math. Phys. \textbf{406}, (2025), Paper No.~293.
	\newblock
	\burlalt{doi:10.1007/s00220-025-05455-7}{https://dx.doi.org/10.1007/s00220-025-05455-7}.
	\newblock
	\burlalt{arXiv:2408.13693}{https://arxiv.org/abs/2408.13693}.


	\bibitem{VW26}
	{\rm K.~Vassilev, B.~Wu}.
	\newblock \emph{Rigorous Derivation of the Wave Kinetic Equation for full $\beta$-FPUT System}.
	\newblock Preprint, (2026).
	\newblock
	\burlalt{arXiv:2605.19308}{https://arxiv.org/abs/2605.19308}.


	\bibitem{Z69}
	{\rm W.~Zimmermann}.
	\newblock \emph{Convergence of Bogoliubov's method of renormalization in momentum space}.
	\newblock Comm. Math. Phys. \textbf{15}, (1969), 208--234.
	\newblock
	\burlalt{doi:10.1007/BF01645676}{https://dx.doi.org/10.1007/BF01645676}.


\end{thebibliography}
\end{document}